\documentclass[11pt,oneside,reqno]{amsart}
\usepackage{geometry}
\usepackage[utf8]{inputenc}
\usepackage{amstext,latexsym,amsbsy,amsmath,amssymb,amsthm,mathtools,relsize,geometry,enumerate}
\usepackage{hyperref}
\hypersetup{
  colorlinks   = true, 
  urlcolor     = blue, 
  linkcolor    = red, 
  citecolor   = red 
}
\theoremstyle{definition}
\usepackage{graphicx}
\setkeys{Gin}{width=\linewidth,totalheight=\textheight,keepaspectratio}
\graphicspath{{./images/}}

\usepackage{multicol}
\usepackage{booktabs}

\usepackage{color}

\newcommand{\Sc}{\mathcal{S}}

\newcommand{\rbb}{\mathbb{R}}

\newcommand{\W}{\mathcal{W}}
\renewcommand{\H}{\mathcal{H}}

\newcommand{\la}{\langle}
\newcommand{\ra}{\rangle}

\newcommand{\xbar}{\overline{x}}
\newcommand{\vbar}{\overline{v}}

\newcommand{\mi}{\wedge}

\newcommand{\f}{\varphi}

\newcommand{\E}[1]{\mathbb{E}\left[#1\right]}

\newcommand{\Enone}[1]{\mathbb{E}#1}

\renewcommand{\P}[1]{\mathbb{P}\left\{#1\right\}}
\newcommand{\Pnone}{\mathbb{P}}

\newcommand{\wt}[1]{ \widetilde{#1} }

\newcommand{\Hs}{\mathcal{H}_{-s}}
\newcommand{\Hsdot}{\dot{\mathcal{H}}_{-s}}

\theoremstyle{definition}
\theoremstyle{plain}
\newtheorem{theorem}{Theorem}[section]

\newtheorem{lemma}[theorem]{Lemma}
\newtheorem{assumption}[theorem]{Assumption}
\newtheorem{proposition}[theorem]{Proposition}

\newtheorem{remark}[theorem]{Remark}

\numberwithin{equation}{section}

\title{The small-mass limit and white-noise limit of an infinite dimensional Generalized Langevin Equation}

\author{Hung D.~Nguyen$^1$}

\thanks{\noindent \hspace{-0.52cm} $^1$ Department of Mathematics, Tulane University. 6823 St Charles Ave, New Orleans, LA, 70118.}

\begin{document}

\maketitle

%

\begin{abstract}
We study asymptotic properties of the Generalized Langevin Equation (GLE) in the presence of a wide class of external potential wells with a power-law decay memory kernel. When the memory can be expressed as a sum of exponentials, a class of Markovian systems in infinite-dimensional spaces is used to represent the GLE. The solutions are shown to converge in probability in the small-mass limit and the white-noise limit to appropriate systems under minimal assumptions, of which no global Lipschitz condition is required on the potentials. With further assumptions about space regularity and potentials, we obtain $L^1$ convergence in the white-noise limit.
\end{abstract}

\section{Introduction}
The Generalized Langevin Equation is a Stochastic Integro-Differential Equation that is commonly used to model the velocity $\{v(t)\}_{t \geq 0}$ of a microparticle in a thermally fluctuating viscoelastic fluid \cite{mason1995optical,kou2008stochastic,hohenegger2017fluid}. It can be written in the following form
\begin{equation}\label{eqn:GLE}
\begin{aligned}
\dot{x}(t)&=v(t),\\
m\, \dot{v}(t)&=-\gamma v(t)-\Phi'(x(t))-\int_{-\infty}^t\!\!K(t-s)v(s)\, ds+F(t)+\sqrt{2\gamma} \, \dot{W}(t),
\end{aligned}
\end{equation}
where $m > 0$ is the particle mass, $\gamma>0$ is the viscous drag coefficient, $\Phi$ is a potential well and $K(t)$ is a phenomenological memory kernel that summarizes the delayed drag effects by the fluid on the particle. The noise has two components: $F(t)$ is a mean-zero, stationary Gaussian process with autocovariance $\E{F(t)F(s)}=K(|t-s|)$, and $W(t)$ is a standard two-sided Brownian motion. The appearance of $K(t)$ in the autocovariance of $F(t)$ is a manifestation of the Fluctuation-Dissipation relationship, originally stated in~\cite{kubo1966fluctuation}, see also \cite{pavliotis2014stochastic} for a more systematic review.

When there are no external forces, the GLE has the form
\begin{equation}\label{eqn:GLE:linear}
\begin{aligned}
\dot{x}(t)&=v(t),\\
m\, \dot{v}(t)&=-\gamma v(t)-\int_{-\infty}^t\!\!K(t-s)v(s)\, ds+F(t)+\sqrt{2\gamma} \, \dot{W}(t),
\end{aligned}
\end{equation}
it was shown in~\cite{mckinley2017anomalous} that with extra assumptions, when $K$ is integrable, the \emph{Mean-Squared Displacement} (MSD) $\Enone\, x(t)^2$ satisfies $\Enone\, x(t)^2 \sim t$ as $t\to\infty$; otherwise, if $K(t)\sim t^{-\alpha}$, $\alpha\in(0,1)$, then $\Enone x(t)^2\sim t^{\alpha}$ as $t\to\infty$. The former asymptotic behavior of the MSD is called \emph{diffusive} whereas the latter is called \emph{subdiffusive}. Here the notation $f(t)\sim g(t)$ as $t\to\infty$ means
\begin{align*}
\frac{f(t)}{g(t)}\to c\in(0,\infty),\quad\text{as}\quad t\to\infty.
\end{align*}
It has been observed that when $K(t)$ is written as a sum of exponential functions,  by adding auxiliary terms, the non-Markovian GLE~\eqref{eqn:GLE} can be mapped onto a multi-dimensional Markovian system~ \cite{mori1965continued,zwanzig1973nonlinear,kupferman2002long,kupferman2004fractional}. If $K(t)$ has the form of a finite sum of exponentials, the resulting finite-dimensional SDE was studied extensively in e.g.~\cite{ottobre2011asymptotic,pavliotis2014stochastic}. One can show that these systems admit a unique invariant structure with geometric ergodicity. Moreover, the marginal density of the pair $(x,v)$ is independent of $K(t)$. It is also worthwhile to note that, in the case of the linear GLE~\eqref{eqn:GLE:linear}, these memory kernels $K(t)$ produce \emph{diffusive} MSD since they are integrable~\cite{mckinley2017anomalous}. In order to include memory kernels that have a  power-law decay, one has to consider an infinite sum of exponentials resulting in a corresponding infinite-dimensional system. 

From now on, we shall adopt the notations from~\cite{glatt2018generalized}. Let $\alpha, \beta >0$ be given, and define constants $c_k, \lambda_k$, $k=1,2, \ldots$, by 
\begin{align} \label{c-k}
c_k=\frac{1}{k^{1+\alpha\beta}},\ \lambda_k=\frac{1}{k^\beta}.
\end{align} 
We introduce the memory kernel $K(t)$ given by
\begin{equation}
\label{eqn:K}
K(t)=\sum_{k\geq 1} c_k e^{-\lambda_k t}.
\end{equation}
It is shown that (see Example 3.3 of~\cite{abate1999infinite}) with this choice of constants $c_k$ and $\lambda_k$, $K(t)$ obeys a power-law decay, namely
\begin{equation} \label{lim:K}
K(t) \sim t^{-\alpha} \,\,\, \text{ as } \,\,\, t\rightarrow \infty, 
\end{equation}
where $\alpha$ is as in~\eqref{c-k}. The constant $\beta$ is an auxiliary parameter that is only assumed to be positive. When $\alpha>1$, as mentioned above in the linear GLE~\eqref{eqn:GLE:linear}, $K(t)$ is in the \emph{diffusive} regime, whereas for $\alpha\in(0,1)$, $K$ belongs to the \emph{subdiffusive} regime. There is however no claim regarding to the case $\alpha=1$. For such reason, it is called the \emph{critical regime}. With $K(t)$ defined as in~\eqref{eqn:K}, the GLE~\eqref{eqn:GLE} is expressed as the following infinite-dimensional system \cite{glatt2018generalized}
\begin{equation}\label{eqn:GLE-Markov}
\begin{aligned}
d x(t) &= v(t)\, d t, \\
m\, d v(t)&=\big(\!\!-\gamma v(t)-\Phi'(x(t))-\sum_{k\geq 1} \sqrt{c_k} z_k(t)\big)\,dt+\sqrt{2\gamma}\, dW_0(t), \\
d z_k(t)&=\left(-\lambda_k z_k(t)+ \sqrt{c_k}v(t)\right) \, dt+\sqrt{2\lambda_k}\, dW_k(t),\qquad k\geq 1, 
\end{aligned}	
\end{equation}
where $W_k$ are independent, standard Brownian motions.  In~\cite{glatt2018generalized}, the well-posedness and the existence of invariant structures of~\eqref{eqn:GLE-Markov} were studied for all $\alpha>0$. Employing a recent advance tool called \emph{asymptotic coupling}~\cite{hairer2011asymptotic,glatt2017unique}, it can be shown that~\eqref{eqn:GLE-Markov} admits a unique invariant distribution in the diffusive regime, ($\alpha>1$). However, ergodicity when $\alpha\in(0,1]$ remains an open question.

The goal of this note is to give an analysis of the behavior of~\eqref{eqn:GLE-Markov} in two different limits. 

First, we are interested in the asymptotic behavior of~\eqref{eqn:GLE-Markov} concerning the small-mass limit, namely taking $m$ to zero on the LHS of the second line in~\eqref{eqn:GLE-Markov}. Due to the random perturbations, the velocity $v(t)$ is fluctuating fast whereas the displacement $x(t)$ is still moving slow. We hence would like to find a process $u(t)$ such that on any compact interval $[0,T]$,
\begin{align*}
\lim_{m\downarrow 0}\sup_{0\leq r\leq t}|x(r)-u(r)|=0,
\end{align*}
where the limit holds in an appropriate sense. Such statement is called Smoluchowski-Kramer approximation \cite{freidlin2004some}. There is a literature of analyzing asymptotic behaviors for fast-slow processes when taking zero-mass limit. Earliest results in this direction seem to be the works of~\cite{kramers1940brownian,von1916drei}. For more recent studies in finite dimensional systems, we refer to~\cite{freidlin2004some} in which the convergence in probability is established with constant drag and multiplicative noise. Under stronger assumptions and using appropriate time rescaling, weak convergence is proved in~\cite{pardoux2003poisson} where the friction is also state-dependent. When the potential is assumed to be Lipschitz, one can obtain better results in $L^p$, following the works of~\cite{hottovy2015smoluchowski,lim2017homogenization}. Without such assumption, convergence in probability is established in~\cite{herzog2016small} provided appropriate Lyapunov controls. In addition, limiting systems are observed numerically in \cite{hottovy2012noise,hottovy2012thermophoresis}. Similar analysis in infinite dimensional settings for semi linear wave equations are studied in a series of paper~\cite{cerrai2006smoluchowski,cerrai2006smoluchowski2,cerrai2014smoluchowski,cerrai2016smoluchowski}.  The systems therein are shown to converge to a heat equation under different assumptions about non linear drifts. Motivated by~\cite{herzog2016small,hottovy2015smoluchowski}, in this note, we establish the convergence in probability for~\eqref{eqn:GLE-Markov}, cf. Theorem~\ref{thm:limit:zeromass}. The technique that we employ is inspired by those in~\cite{herzog2016small}. 

Then, we study the white-noise limit of~\eqref{eqn:GLE-Markov}, namely as the random force $F(t)$ in~\eqref{eqn:GLE} converges to a white noise process. Under different conditions on the potential and space regularity, we aim to find a pair of processes $(u(t),p(t))$ that can be approximated by the $(x(t),v(t))$-component in~\eqref{eqn:GLE-Markov}. While there is a rich history on the small-mass limit, the white-noise limit seems to receive less attention. Nevertheless, there have been many works on the asymptotics of deterministic systems with memories. To name a few in this direction, we refer the reader to \cite{conti2006singular,gatti2005navier,grasselli2002uniform}. With regards to the white-noise limit of our system, we establish the convergence in different modes for a wide class of potentials, that are not necessarily Lipschitz or bounded. While the proof of Theorem~\ref{thm:limit:whitenoise:probconverge} concerning probability convergence shares the same arguments with that of Theorem~\ref{thm:limit:zeromass}, the result in Theorem~\ref{thm:limit:whitenoise:L1converge} concerning strong topology requires more work, where we have to estimate a universal bound on the solutions of~\eqref{eqn:GLE-Markov} using appropriate Lyapunov structures, cf. Proposition~\ref{prop:limit:whitenoise:L2bound}. To the best of our knowledge, these results seem to be new in infinite-dimensional stochastic differential equations with memory. Particularly, they (cf. Theorem~\ref{thm:limit:whitenoise:probconverge} and Theorem~\ref{thm:limit:whitenoise:L1converge}) generalize analogous results for finite-dimensional settings in~\cite{ottobre2011asymptotic}, where $K(t)$ has a form of finite sum of exponentials. 

The rest of this paper is organized as follows. We introduce notations and summarize our main results in Section~\ref{sec:results}. The small-mass limit is addressed rigorously in Section \ref{sec:limit:zeromass}. We obtain a formula for the limiting system as a form of a Smoluchowski-Kramers equation. Finally, Section~\ref{sec:limit:whitenoise} studies the white-noise limit.

\section{Summary of Results}

\label{sec:results}
Throughout this work, we will assume that the potential $\Phi$ satisfies the following growth condition \cite{glatt2018generalized}.
\begin{assumption}\label{cond:Phi} $\Phi\in C^\infty(\rbb)$ and there exists a constant $c>0$ such that for all $x\in\rbb$
\begin{align*}
\quad c(\Phi(x)+1)\geq x^2.     
\end{align*}
By adding a positive constant if necessary, we also assume that $\Phi$ is non-negative.
\end{assumption}
A typical class of potentials $\Phi$ that satisfies Assumption~\ref{cond:Phi} is the class of polynomials of even degree whose leading coefficient is positive. Functions that grow faster than polynomials are also included, e.g. $e^{x^2}$.

We now define a phase space for the infinite-dimensional process 
\begin{align*}
X(t)=(x(t), v(t), z_1(t), z_2(t), \ldots).
\end{align*}
Following~\cite{glatt2018generalized}, let $\H_{-s}$, $s\in \rbb$ denote the Hilbert space given by
 \begin{equation}\label{eqn:H_p}
\H_{-s}=\Big\{X=(x,v,Z)=(x,v,z_1, z_2, \ldots):\|X\|^2_{-s}=x^2+v^2+\sum_{k\geq 1}k^{-2s}z_k^2<\infty\Big\}.
\end{equation}
endowed with the usual inner product $\la\cdot,\cdot\ra_{-s}$,
\begin{equation} \label{eqn:H-inner-prod}
\la X,\widetilde{X}\ra_{-s} = x\wt{x}+v\wt{v}+\sum_{k\geq 1}k^{-2s}z_k\wt{z}_k.
\end{equation}

With regards to kernel parameters $\alpha, \beta$ cf.~\eqref{c-k}, \eqref{eqn:K} and the phase space regularity parameter $s$, we assume that they satisfy the following condition.  
\begin{assumption}\label{cond:wellposed} Let $\alpha,\beta>0$ be as in~\eqref{c-k} and $s\in\rbb$ as in~\eqref{eqn:H_p}. We assume that they satisfy either the \emph{asymptotically diffusive} condition  
\begin{enumerate}
\item[\emph{(D)}]\label{cond:diffusion} $\displaystyle \alpha>1,\, \beta>\frac{1}{\alpha-1}$ and $\displaystyle 1< 2s < (\alpha-1)\beta$;
\end{enumerate}  or the \emph{asymptotically subdiffusive} condition 
\begin{enumerate}
\item[\emph{(SD)}]\label{cond:subdiffusion} $\displaystyle 0<\alpha <1,\, \beta>\frac{
1}{\alpha}$ and $\displaystyle 1< 2s < \alpha\beta$;
\end{enumerate}
or the \emph{critical regime} condition
\begin{enumerate}
\item[\emph{(C)}] $\displaystyle \alpha =1,\, \beta>1$ and $\displaystyle 1< 2s < \beta$;
\end{enumerate}
\end{assumption} 

Under Assumption~\ref{cond:Phi} and Assumption~\ref{cond:wellposed}, the well-posedness of~\eqref{eqn:GLE-Markov} was studied rigorously in~\cite{glatt2018generalized}. 

\subsection{Small-mass Limit}
 
In regards to the small-mass limit ($m\to 0$), we introduce the following limiting system whose derivation will be explained later at the end of this subsection
\begin{equation} \label{eqn:GLE-Markov:limit:zeromass}
\begin{aligned}
\gamma du(t)&= \Big(\!\!-\Phi'(u(t))-\Big(\sum_{k\geq 1}c_k\Big) u(t)-\sum_{k\geq 1}\sqrt{c_k}f_k(t)\Big)dt+\sqrt{2\gamma}dW_0(t),\\
df_k(t)&=\big(\!\!-\lambda_k f_k(t)-\lambda_k\sqrt{c_k}u(t)\big)dt+\sqrt{2\lambda_k}dW_k(t),\quad k=1,2\dots
\end{aligned}
\end{equation}
The new phase space for the solution $U(t)=(u(t),f_1(t),f_2(t),\dots)$ of~\eqref{eqn:GLE-Markov:limit:zeromass} is denoted by $\Hsdot$, $s\in \rbb$,  
 \begin{equation}\label{eqn:Hsdot}
\Hsdot=\Big\{U=(u,f_1, f_2, \ldots):\|U\|_{\Hsdot}^2=u^2+\sum_{k\geq 1}k^{-2s}f_k^2<\infty\Big\},
\end{equation}
endowed with the usual inner product. We can regard $\Hsdot$ as a subspace of $\Hs$ whose $v-$component is equal to zero. Recalling $c_k$ in~\eqref{c-k}, it is straightforward to see that if $(x,v,z_1,z_2\dots)\in\Hs$ then $(x,z_1-\sqrt{c_1}x,z_2-\sqrt{c_2}x,\dots)\in\Hsdot$.

From now on, we shall fix a stochastic basis $\mathcal{S}=\left(\Omega,\mathcal{F},\Pnone,\{\mathcal{F}_t\}_{t\geq 0},W\right)$ satisfying the usual conditions \cite{karatzas2012brownian}. Here $W$ is the cylindrical Wiener process defined on an auxiliary Wiener space $\W$ with the usual decomposition
\begin{align*}
W(t)=e^{\W}_0W_0(t)+e^{\W}_1 W_1(t)+\dots,
\end{align*}
where $\{e^{\W}_0,e^{\W}_1,\dots\}$ is the canonical basis of $\W$, and $\{W_k(t)\}_{k\geq 0}$ are independent one-dimensional Brownian Motions  \cite{da2014stochastic}. The well-posedness of~\eqref{eqn:GLE-Markov:limit:zeromass} is guaranteed by the following result.
\begin{proposition} \label{prop:wellposed:limit:zeromass} Suppose that $\Phi$ satisfies Assumption \ref{cond:Phi} and the constants $\alpha, \beta, s$ satisfy Assumption \ref{cond:wellposed}. Then for all initial conditions $U_0\in \Hsdot$, there exists a unique pathwise solution $U(\cdot,U_0):\Omega\times[0,\infty)\to\Hsdot$ of~\eqref{eqn:GLE-Markov:limit:zeromass} in the following sense: $U(\cdot,U_0)$ is $\mathcal{F}_t$-adapted, $U(\cdot,U_0)\in C([0,\infty),\Hsdot)$ almost surely and that if $\wt{U}(\cdot,U_0)$ is another solution then for every $T\geq 0$,
\begin{displaymath}
\P{\forall t\in[0,T], U(t,U_0)=\wt{U}(t,U_0)}=1.
\end{displaymath}
Moreover, for every $U_0 \in \Hsdot$ and $T\geq 0$, there exists a constant $C(T,U_0)>0$ such that
 \begin{equation} \label{ineq:strong-sol'n}
 \Enone\sup_{0\leq t\leq T }\|U(t)\|_{\Hsdot}^2\leq C(T,U_0).
 \end{equation}
\end{proposition}
The proof of Proposition~\ref{prop:wellposed:limit:zeromass} is quite standard, similar to that of Proposition 6 in~\cite{glatt2018generalized} and will be briefly explained in Section~\ref{sec:limit:zeromass}. We now state the main result concerning the small-mass limit. 
\begin{theorem} \label{thm:limit:zeromass}
Suppose that $\Phi$ satisfies Assumption~\ref{cond:Phi} and the constants $\alpha, \beta, s$ satisfy Assumption~\ref{cond:wellposed}. Let $X_m(t)=(x_m(t),v_m(t),z_{1,m}(t),\dots)$ solve~\eqref{eqn:GLE-Markov} with initial conditions 
\begin{align*}
(x_m(0),v_m(0),z_{1,m}(0),z_{2,m}(0)\dots)=(x,v,z_1,z_2,\dots)\in\Hs,
\end{align*}
and $U(t)=(u(t),f_1(t),\dots)$ solve~\eqref{eqn:GLE-Markov:limit:zeromass} with initial conditions 
\begin{align*}
(u(0),f_1(0),f_2(0)\dots)=(x,z_1-\sqrt{c_1}x,z_2-\sqrt{c_2}x,\dots)\in\Hsdot.
\end{align*}
Then, for every $T,\,\xi>0$, it holds that
\begin{align*}
\Pnone\Big\{\sup_{0\leq t\leq T}|x_m(t)-u(t)|>\xi\Big\}\rightarrow 0,\quad m\rightarrow 0.
\end{align*}
\end{theorem}

It is worthwhile to note that the small-mass limit in Theorem~\ref{thm:limit:zeromass} holds for all regimes ($\alpha>0$) as stated in Assumption~\ref{cond:wellposed}. We will see later that for the white-noise limit, the result is limited to the diffusive regime, namely $\alpha>1$, see Theorem~\ref{thm:limit:whitenoise:probconverge} below. We finish this subsection by a heuristic argument explaining how we derive the limiting system~\eqref{eqn:GLE-Markov:limit:zeromass} and its initial conditions as stated in Theorem~\ref{thm:limit:zeromass}. Following~\cite{herzog2016small}, to determine the limiting system, one may formally set $m=0$ on the RHS of the second equation in \eqref{eqn:GLE-Markov} and substitute $v(t)$ by $dx(t)$ from the first equation to obtain
\begin{align*}
\gamma\, dx(t)=\big[\!\!-\Phi'(x(t))-\sum_{k\geq 1}\sqrt{c_k}z_k(t)\big]dt+\sqrt{2\gamma}dW_0(t).
\end{align*}
The equation on $z_k(t)$ in~\eqref{eqn:GLE-Markov} still depends on $v(t)$, but this can be circumvented by using Duhamel's formula,
\begin{align*}
z_k(t)=e^{-\lambda_k t}z_k(0)+\sqrt{c_k}\int_0^t e^{-\lambda_k(t-r)}v(r)dr+\sqrt{2\lambda_k}\int_0^t e^{-\lambda_k(t-r)}dW_k(r).
\end{align*}
By an integration by parts, we can transform the integral term involving $v(r)$ to
\begin{align*}
\int_0^t e^{-\lambda_k(t-r)}v(r)dr = x(t)-e^{-\lambda_k t}x(0)-\lambda_k\int_0^t e^{-\lambda_k(t-r)}x(r)dr.
\end{align*}
Plugging back into the formula for $z_k(t)$, we find
\begin{align*}
z_k(t)&=e^{-\lambda_k t}(z_k(0)-\sqrt{c_k}x(0))+\sqrt{c_k}x(t)-\lambda_k\sqrt{c_k}\int_0^t e^{-\lambda_k(t-r)}x(r)dr\\
&\qquad\qquad\qquad+\sqrt{2\lambda_k}\int_0^t e^{-\lambda_k(t-r)}dW_k(r).
\end{align*}
We now set $f_k(t):=z_k(t)-\sqrt{c_k}x(t)$ and $u(t):=x(t)$ and thus arrive at~\eqref{eqn:GLE-Markov:limit:zeromass} with the new shifted initial conditions as in Theorem~\ref{thm:limit:zeromass}.

\subsection{White-noise Limit}
Next, we present our results on the asymptotical behavior of the $(x(t),v(t))$-component of~\eqref{eqn:GLE-Markov} in the diffusive regime by an appropriate scaling on the memory kernel $K(t)$ in~\eqref{eqn:K}. For $\epsilon>0$, we introduce $K_\epsilon(t)$ given by
\begin{equation}\label{eqn:K:epsilon}
K_\epsilon(t)=\frac{1}{\epsilon}K\Big(\frac{t}{\epsilon}\Big)=\sum_{k\geq 1}\frac{c_k}{\epsilon}e^{-\frac{\lambda_k}{\epsilon}|t|}.
\end{equation}
With $K_\epsilon$ defined above, the corresponding system~\eqref{eqn:GLE-Markov} becomes
\begin{equation}\label{eqn:GLE-Markov:whitenoise:epsilon}
\begin{aligned}
d x_\epsilon(t) &= v_\epsilon(t)\, d t, \\
m\, d v_\epsilon(t)&=\Big(\!\!-\gamma v_\epsilon(t)-\Phi'(x_\epsilon(t))-\sum_{k\geq 1} \sqrt{\frac{c_k}{\epsilon}} z_{k,\epsilon}(t)\Big)\,dt+\sqrt{2\gamma}\, dW_0(t), \\
d z_{k,\epsilon}(t)&=\Big(\!\!-\frac{\lambda_k}{\epsilon} z_{k,\epsilon}(t)+ \sqrt{\frac{c_k}{\epsilon}}v_\epsilon(t)\Big) \, dt+\sqrt{\frac{2\lambda_k}{\epsilon}}\, dW_k(t),\qquad k\geq 1.
\end{aligned}	
\end{equation}
Inspired by~\cite{ottobre2011asymptotic}, we consider the following system
\begin{equation}\label{eqn:GLE-Markov:limit:whitenoise}
\begin{aligned}
d u(t) &= p(t)\, d t ,\\
m\, d p(t)&=\Big(\!\!\!-\Big(\gamma+\sum_{k\geq 1}\frac{c_k}{\lambda_k}\Big) p(t)-\Phi'(u(t))\Big)\,dt\\
&\qquad\qquad-\sum_{k\geq 1}\sqrt{\frac{2c_k}{\lambda_k}}dW_k(t)+\sqrt{2\gamma}\, dW_0(t).
\end{aligned}	
\end{equation}
The well-posedness of~\eqref{eqn:GLE-Markov:limit:whitenoise} will be addressed briefly in Section~\ref{sec:limit:whitenoise}. 
We then assert that $(x_\epsilon(t),v_\epsilon(t))$ converges to the solution $(u(t),p(t))$ of~\eqref{eqn:GLE-Markov:limit:whitenoise} in the following sense.

\begin{theorem} \label{thm:limit:whitenoise:probconverge}
Suppose that $\Phi$ satisfies Assumption \ref{cond:Phi} and the constants $\alpha, \beta, s$ satisfy Condition (D) of Assumption \ref{cond:wellposed}. Let $X_\epsilon(t)=(x_\epsilon(t),v_\epsilon(t),z_{1,\epsilon}(t),\dots)$ be the solution of~\eqref{eqn:GLE-Markov:whitenoise:epsilon} with initial conditions 
\begin{align*}(x_\epsilon(0),v_\epsilon(0),z_{1,\epsilon}(0),z_{2,\epsilon}(0),\dots)=(x,v,z_{1},z_2,\dots)\in\Hs,
\end{align*} 
and $(u(t),p(t))$ be the solution of~\eqref{eqn:GLE-Markov:limit:whitenoise} with initial conditions $(u(0),p(0))=(x,v)$. Then, for every $T,\,\xi>0$,
\begin{align*}
\Pnone\Big\{\sup_{0\leq t\leq T}|x_\epsilon(t)-u(t)|+|v_\epsilon(t)-p(t)|>\xi\Big\}\rightarrow 0,\quad \epsilon\rightarrow 0.
\end{align*}
\end{theorem}

The reader may wonder why the convergence result of Theorem~\ref{thm:limit:whitenoise:probconverge} is restricted to the \emph{diffusive} regime, namely $\alpha>1$ according to Condition (D) of Assumption~\ref{cond:wellposed}. Heuristically, since the memory kernel $K(t)$ decays like $t^{-\alpha}$ as $t\to\infty$, we see that
\begin{align*}
K_\epsilon(t)=\frac{1}{\epsilon}K\Big(\frac{t}{\epsilon}\Big)\sim \epsilon^{\alpha-1} t^{-\alpha},\quad t\to\infty.
\end{align*}
By shrinking $\epsilon$ further to zero, if $\alpha>1$, $K_\epsilon(t)$ does not behave ``badly" at infinity. In fact, it is not difficult to show that as $\epsilon\downarrow 0$, $K_\epsilon$ converges to the Dirac function $\delta_0$ centered at the origin, in the sense of tempered distribution, namely, for every $\f\in\Sc$, the Schwartz space on $\rbb$, it holds that
\begin{align*}
\int_\rbb K_\epsilon(t)\f(t)dt\to |K|_{L^1(\rbb)}\f(0).
\end{align*}
which implies that the random force $F(t)$ in~\eqref{eqn:GLE} converges to a white noise process in the sense of random distribution,   cf. \cite{ito1954stationary}, hence the so called ``white-noise limit". We will see later in the proof of Proposition~\ref{prop:limit:whitenoise:L2bound} that the condition $\alpha>1$ is crucial for our analysis in order to obtain bounds on the solutions.

Finally, if $\Phi$ and parameters $\alpha,\, \beta$ satisfy stronger assumptions, then we are able to obtain better convergence than the result in Theorem~\ref{thm:limit:whitenoise:probconverge}. To be precise, we assume the following condition on $\Phi$.
\begin{assumption} \label{cond:Phi:whitenoise} There exist constants $n,\, c>0$ such that for every $x,\, y\in\rbb$,
\begin{displaymath}
\Phi'(x)y\leq c(\Phi(x)+|y|^n+1).
\end{displaymath}
\end{assumption}
The assumption above is again a requirement about the growth of $\Phi'$ that guarantees a universal bound independent of $\epsilon$ on the solution $(x_\epsilon(t),v_\epsilon(t))$ of~\eqref{eqn:GLE-Markov:whitenoise:epsilon}, cf. Proposition~\ref{prop:limit:whitenoise:L2bound}. We remark that a function $\Phi$ satisfying Assumption~\ref{cond:Phi:whitenoise} need not satisfy Assumption~\ref{cond:Phi}, taking $\Phi$ a constant for example. It is also worthwhile to note that the class of polynomials of even degree satisfies Assumption~\ref{cond:Phi:whitenoise}. However, functions growing exponentially fast, e.g. $e^{x^2}$, do not. 

With regards to space regularities, we assume the following condition about parameters $\alpha,\,\beta$.
\begin{assumption} \label{cond:whitenoise:L1converge}  Let $\alpha,\beta>0$ be as in~\eqref{c-k}. We assume that they satisfy 
\begin{align*}
\alpha>2,\quad\text{and}\quad (\alpha-2)\beta>1.
\end{align*}
\end{assumption}

We then have the following important result.
\begin{theorem} \label{thm:limit:whitenoise:L1converge}
Suppose that $\Phi$ satisfies Assumption \ref{cond:Phi} and Assumption~\ref{cond:Phi:whitenoise} and that the constants $\alpha, \beta, s$ satisfy Condition (D) of Assumption \ref{cond:wellposed} and Assumption~\ref{cond:whitenoise:L1converge}. Let $X_\epsilon(t)=(x_\epsilon(t),v_\epsilon(t),z_{1,\epsilon}(t),\dots)$ be the solution of~\eqref{eqn:GLE-Markov:whitenoise:epsilon} with initial conditions 
\begin{align*}(x_\epsilon(0),v_\epsilon(0),z_{1,\epsilon}(0),z_{2,\epsilon}(0),\dots)=(x,v,z_{1},z_2,\dots)\in\Hs,
\end{align*}
 and $(u(t),p(t))$ be the solution of~\eqref{eqn:GLE-Markov:limit:whitenoise} with initial conditions $(u(0),p(0))=(x,v)$. Then, for every $T>0$, $1\leq q<2$, it holds that
\begin{equation*}
\Enone\Big[\sup_{0\leq t\leq T}|x_\epsilon(t)-u(t)|^q+|v_\epsilon(t)-p(t)|^q\Big]\rightarrow 0, \quad \epsilon\to 0.
\end{equation*}
\end{theorem}
Theorem~\ref{thm:limit:whitenoise:L1converge} strengthens a previous result from \cite{ottobre2011asymptotic}, where the potential $\Phi'$ is assumed to be bounded. The proofs of Theorem~\ref{thm:limit:whitenoise:probconverge} and Theorem~\ref{thm:limit:whitenoise:L1converge} will be carried out in Section~\ref{sec:limit:whitenoise}.

\section{Zero-mass limit}\label{sec:limit:zeromass}

Throughout the rest of the paper, $C,\,c$ denote generic positive constants. The important parameters that they depend on will be indicated in parenthesis, e.g. $c(T,q)$ depends on parameters $T$ and $q$.

In this section, for notation simplicity, we shall omit the subscript $m$ in 
\begin{align*}
X_m(t)=(x_m(t),v_m(t),z_{1,m}(t),\dots).
\end{align*}

We begin by addressing the well-posedness of~\eqref{eqn:GLE-Markov:limit:zeromass} whose proof follows a standard Lyapunov-type argument that was also used to establish the well-posedness of~\eqref{eqn:GLE-Markov} in \cite{glatt2018generalized}. The technique is classical and has been employed previously in literature  \cite{albeverio2008spde,glatt2009strong,jacod2006calcul}. We shall omit specific details and briefly summarize the main steps. 
\begin{proof}[Sketch of the proof of Proposition \ref{prop:wellposed:limit:zeromass}]
For $R>0$, let $\theta^R\in C^\infty(\rbb , [0,1])$  satisfy
\begin{align} \label{defn:theta-R}
\theta^R(x) = \begin{cases}
1 & \text{ if } |x| \leq R, \\
0 & \text{ if } |x| \geq R+1.
\end{cases}
\end{align}
We consider the ``cutoff" equation corresponding to~\eqref{eqn:GLE-Markov:limit:zeromass}
\begin{equation}\label{eqn:GLE-Markov:limit:zeromass:cutoff}
\begin{aligned}
\gamma du(t)&= \Big(\!\!-\Phi'(u(t))\theta^R(u(t))-\Big(\sum_{k\geq 1}c_k\Big) u(t)-\sum_{k\geq 1}\sqrt{c_k}f_k(t)\Big)dt+\sqrt{2\gamma}dW_0(t),\\
df_k(t)&=\big(\!\!-\lambda_k f_k(t)-\lambda_k\sqrt{c_k}u(t)\big)dt+\sqrt{2\lambda_k}dW_k(t),\quad k=1,2\dots
\end{aligned}
\end{equation}
We observe that in~\eqref{eqn:GLE-Markov:limit:zeromass:cutoff}, the drift term is globally Lipschitz and the noise is additive. Thus, by using a standard Banach fixed point argument, the corresponding global (in time) solution $U^R$ exists and is unique. Next, define the stopping time 
\[\tau_R=\inf\left\{t>0: \|U(t)\|_{\Hsdot}>R\right\}.\]
Note that, for all times $t< \tau_R$, $U^R$ solves \eqref{eqn:GLE-Markov:limit:zeromass}. Consequently, the solution \eqref{eqn:GLE-Markov:limit:zeromass} exists and is unique up until the \emph{time of explosion} $\tau_\infty=\lim_{R\to\infty}\tau_R$, which is possibly finite on a set of positive probability. We finally introduce the Lyapunov function 
\begin{equation} \label{eqn:Lyapunov:zeromass}
\Psi(U):=\frac{1}{\gamma}\bigg(\Phi(u)+\Big(\sum_{k\geq 1}c_k\Big)\frac{ u^2}{2}\bigg)+\frac{1}{2}\sum_{k \geq 1}k^{-2s}f_k^2.
\end{equation}
It is clear that $\Psi(U)$ dominates $\|U\|_{\Hsdot}^2$. Applying Ito's formula to $\Psi(U)$, one can derive a global bound on the solutions $U^R(t)$ that is independent of $R$, namely, there exists a constant $C(U_0,T)$ such that
\begin{equation*}
\Enone\Big[\sup_{0\leq t\leq T} \|U(t\wedge\tau_R)\|_{\Hsdot}^2 \Big]\leq C(T,U_0).
\end{equation*}
Sending $R$ to infinity, it follows from Fatou's Lemma that 
\begin{equation*}
\Enone\Big[\sup_{0\leq t\leq T} \|U(t\wedge\tau_\infty)\|_{\Hsdot}^2 \Big]\leq C(U_0,T).
\end{equation*}
implying $\P{T<\tau_\infty}=1$ for any $T>0$. Taking $T$ to infinity, we see that $\P{\tau_\infty=\infty}=1$, thereby obtaining the global solution of~\eqref{eqn:GLE-Markov:limit:zeromass}.
\end{proof}

Although the construction of the global solution $U(t)$ of~\eqref{eqn:GLE-Markov:limit:zeromass} via the local solutions $U^R(t)$ of~\eqref{eqn:GLE-Markov:limit:zeromass:cutoff} is quite standard, the proof of Theorem~\ref{thm:limit:zeromass} will make use of a non trivial observation on these local solutions. The arguments that we are going to employ are inspired from the work of~\cite{herzog2016small}. Before diving into detail, we briefly explain the main idea, which is a two-fold: first, we show that the result holds for $\Phi'$ being Lipschitz. In particular, we obtain the convergence in sup norm for the local solutions, namely for all $R,\,T>0$, we have
\begin{align*}
\Enone\Big[\sup_{0\leq t\leq T}\big|x^R(t)-u^R(t)\big|\Big]\rightarrow 0,\quad m\rightarrow 0,
\end{align*}
where $x^R(t)$ is in the following cut-off system for~\eqref{eqn:GLE-Markov}
\begin{equation}\label{eqn:GLE-Markov:cutoff}
\begin{aligned}
d x(t) &= v(t)\, d t, \\
m\, d v(t)&=\big(-\gamma v(t)-\Phi'(x(t))\theta^R(x(t))-\sum_{k\geq 1} \sqrt{c_k} z_k(t)\big)\,dt+\sqrt{2\gamma}\, dW_0(t), \\
d z_k(t)&=\left(-\lambda_k z_k(t)+ \sqrt{c_k}v(t)\right) \, dt+\sqrt{2\lambda_k}\, dW_k(t),\qquad k\geq 1.
\end{aligned}	
\end{equation}
Then, by taking $R$ necessarily large, we obtain the desired result.

We now proceed by showing that the result holds true in a simpler setting where $\Phi'$ is globally Lipschitz. The proof is adapted from that of Theorem 1 of~\cite{hottovy2015smoluchowski}. 
\begin{proposition}\label{prop:limit:zeromass:lipschitz} Suppose that that $\Phi'$ is globally Lipschitz and that the constants $\alpha, \beta, s$ satisfy Assumption \ref{cond:wellposed}. Let $X(t)=(x(t),v(t),z_{1}(t),\dots)$ solve~\eqref{eqn:GLE-Markov} with initial conditions
\begin{align*}
(x(0),v(0),z_{1}(0),\dots)=(x,v,z_1,\dots)\in\Hs,
\end{align*}
and $U(t)=(u(t),f_1(t),\dots)$ solve~\eqref{eqn:GLE-Markov:limit:zeromass} with initial conditions 
\begin{align*}
(u(0),f_1(0),f_2(0)\dots)=(x,z_1-\sqrt{c_1}x,z_2-\sqrt{c_2}x,\dots)\in\Hsdot. 
\end{align*}
Then, for every $T,q>0$, it holds that
\begin{align*}
\Enone\sup_{0\leq t\leq T}\big|x(t)-u(t)\big|^q\rightarrow 0,\quad m\rightarrow 0.
\end{align*}
\end{proposition}  
\begin{proof} Using Duhamel's formula, $z_{k}(t)$ from~\eqref{eqn:GLE-Markov} can be solved explicitly as
\begin{align}\label{eqn:limit:zeromass:lipschitz:0}
z_{k}(t) = e^{-\lambda_k t}z_k(0)+\sqrt{c_k}\int_0^t e^{-\lambda_k(t-r)}v(r)dr+\sqrt{2\lambda_k}\int_0^t e^{-\lambda_k(t-r)}dW_k(r),
\end{align}
which is equivalent to
\begin{align*}
z_{k}(t)& = e^{-\lambda_k t}(z_k(0)-\sqrt{c_k}x(0))+\sqrt{c_k}x(t)-\sqrt{c_k}\lambda_k\int_0^t e^{-\lambda_k(t-r)}x(r)dr\\
&\qquad\qquad+\sqrt{2\lambda_k}\int_0^t e^{-\lambda_k(t-r)}dW_k(r),
\end{align*}
where we have used an integration by parts on the term $\int_0^t e^{-\lambda_k(t-r)}v(r)dr$ in the first equality. Substituting into the second equation of~\eqref{eqn:GLE-Markov}, we arrive at
\begin{equation}\label{eqn:limit:zeromass:lipschitz:1}
\begin{aligned} 
\MoveEqLeft[2]m\,dv(t)+\gamma\, dx(t)\\
 &= \bigg(\!\!-\Phi'(x(t))-\sum_{k\geq 1}\sqrt{c_k}e^{-\lambda_k t}(z_k(0)-\sqrt{c_k}x(0))-\Big(\sum_{k\geq 1}c_k\Big)x(t)\\
&\qquad+\sum_{k\geq 1}c_k\lambda_k\int_0^t e^{-\lambda_k(t-r)}x(r)dr -\sum_{k\geq 1}\sqrt{2c_k\lambda_k}\int_0^t e^{-\lambda_k(t-r)}dW_k(r)\bigg)dt\\
&\qquad+\sqrt{2\gamma}dW_0(t).
\end{aligned}
\end{equation}
Likewise, we obtain the following equation from~\eqref{eqn:GLE-Markov:limit:zeromass}
\begin{equation}\label{eqn:limit:zeromass:lipschitz:2}
\begin{aligned} 
 \MoveEqLeft[2] \gamma\, du(t)\\
 &= \bigg(\!\!-\Phi'(u(t))-\Big(\sum_{k\geq 1}c_k\Big)u(t)-\sum_{k\geq 1}\sqrt{c_k}e^{-\lambda_k t}(z_k(0)-\sqrt{c_k}x(0))\\
&\qquad+\sum_{k\geq 1}c_k\lambda_k\int_0^t e^{-\lambda_k(t-r)}u(r)dr -\sum_{k\geq 1}\sqrt{2c_k\lambda_k}\int_0^t e^{-\lambda_k(t-r)}dW_k(r)\bigg)dt\\
&\qquad+\sqrt{2\gamma}dW_0(t).
\end{aligned}
\end{equation}
Subtracting~\eqref{eqn:limit:zeromass:lipschitz:2} from~\eqref{eqn:limit:zeromass:lipschitz:1} and setting $\xbar(t)=x(t)-u(t)$, we find that
\begin{align*}
\MoveEqLeft[3]m\,dv(t)+\gamma\, d\xbar(t)\\
&= \bigg(\!\!-\big[\Phi'(x(t))-\Phi'(u(t))\big]-\Big(\sum_{k\geq 1}c_k\Big)\xbar(t)+\sum_{k\geq 1}c_k\lambda_k\int_0^t e^{-\lambda_k(t-r)}\xbar(r)dr \bigg)dt\\
&\leq c\Big(1+\sum_{k\geq 1}c_k\Big)\sup_{0\leq r\leq t}|\xbar(r)|dt,
\end{align*}
where $c>0$ is a Lipschitz constant for $\Phi'$. Recalling $c_k$ from~\eqref{c-k}, we apply Gronwall's inequality to estimate for all $T,\,q>0$
\begin{align*}
\Enone\sup_{0\leq t\leq T}|\xbar(t)|^q\leq m^q\Enone\sup_{0\leq t \leq T}|v(t)-v|^q\,e^{c(T)}.
\end{align*}
The result now follows immediately from Proposition \ref{prop:limit:zeromass:lipschitz:1} below.
\end{proof}

\begin{proposition}\label{prop:limit:zeromass:lipschitz:1} Under the same Hypothesis of Theorem~\ref{thm:limit:zeromass}, suppose further that $\Phi'(x)$ is globally Lipschitz. Let $X(t)=(x(t),v(t),z_{1}(t),\dots)$ solve~\eqref{eqn:GLE-Markov} with initial conditions $(x(0),v(0),z_{1}(0),\dots)=(x,v,z_1,\dots)\in\Hs$. Then, for every $T>0,\,q>1$, it holds that
\begin{align*}
m^q\Enone\sup_{0\leq t\leq T}|v(t)|^q\to 0,\quad m\to 0.
\end{align*}
\end{proposition} 

In order to prove Proposition~\ref{prop:limit:zeromass:lipschitz:1}, we need the following important lemma whose proof is based on Lemma 3.19, \cite{blount1991comparison} and Lemma 2, \cite{hottovy2015smoluchowski}. It will be also useful later in Section~\ref{sec:limit:whitenoise}.
\begin{lemma}\label{lem:limit:zeromass} Given $\kappa,\,\eta>0$, let $f(t) = \sqrt{2\kappa}\int_0^t e^{-\eta(t-r)}dW(r)$ where $W(t)$ is a standard Brownian Motion. Then, for all $T>0,\, q> 1$, there exists a constant $C(T,q)>0$ such that
\begin{align}\label{ineq:limit:zeromass:1}
\Enone\sup_{0\leq t\leq T}f(t)^{2q}\leq \frac{\kappa^q}{\eta^{q-1}} C(T,q).
\end{align}
\end{lemma}
\begin{remark} The estimate in~\eqref{ineq:limit:zeromass:1} is sharper than the usual exponential martingale estimate. In finite-dimensional settings, it is sufficient to bound the LHS of~\eqref{ineq:limit:zeromass:1} by $C(T,q,\eta,\kappa)$, cf.~\cite{blount1991comparison,hottovy2015smoluchowski,lim2017homogenization}. In our setting, we have to keep track explicitly in term of $\eta$ and $\kappa$, hence the RHS of~\eqref{ineq:limit:zeromass:1}. 
\end{remark}
The proof of Lemma~\ref{lem:limit:zeromass} is similar to that of Lemma 2,~\cite{hottovy2015smoluchowski}. We include it here for the sake of completeness.
\begin{proof}[Proof of Lemma~\ref{lem:limit:zeromass}] 
In view of Lemma 3.19, \cite{blount1991comparison}, we have the following estimate for $A>0$,
\begin{align*}
\Pnone\Big\{\sup_{0\leq t\leq T}f(t)^2\geq A\Big\}\leq \frac{\eta T}{\int_0^{\sqrt{\eta A/\kappa}}e^{r^2/2}\int_0^r e^{-\ell^2/2}d\ell\,dr}.
\end{align*}
We proceed to find a lower bound for the above denominator. To this end, we first claim that for $r\geq 0$,
\begin{align*}
\int_0^r e^{-\ell^2/2}d\ell \geq \frac{re^{-r^2/4}}{2}.
\end{align*}
Indeed, on one hand, if $r\geq 1$, then
\begin{align*}
e^{r^2/4}\int_0^r e^{-\ell^2/2}d\ell \geq e^{r^2/4}\int_0^1 e^{-\ell^2/2}d\ell\geq e^{r^2/4}\int_0^1 e^{-1/2}d\ell\geq \frac{e^{r^2/4}}{2}\geq \frac{r}{2}.
\end{align*}
On the other hand, if $0\leq r\leq 1$, then
\begin{align*}
e^{r^2/4}\int_0^r e^{-\ell^2/2}d\ell\geq \int_0^r 1-\frac{\ell^2}{2} d\ell=r\big(1-\frac{r^2}{6}\big)\geq \frac{r}{2}.
\end{align*}
With this observation, we find
\begin{align*}
\int_0^{\sqrt{\eta A/\kappa}}e^{r^2/2}\int_0^r e^{-\ell^2/2}d\ell\,dr \geq \int_0^{\sqrt{\eta A/\kappa}}e^{r^2/2}\frac{r e^{-r^2/4}}{2}dr=e^{\eta A/4\kappa}-1\geq \frac{\eta A}{4\kappa} e^{\eta A/8\kappa},
\end{align*}
where in the last implication, we have used the following inequality for every $r\geq 0$, 
\begin{align*}
e^r-1\geq re^{r/2}.
\end{align*}
Putting everything together, we obtain
\begin{align*}
\Pnone\bigg\{\sup_{0\leq t\leq T}f(t)^2\geq A\bigg\}\leq \frac{\eta T}{\frac{\eta A}{4\kappa} e^{\eta A/8\kappa}}=\frac{4\kappa T}{A}e^{-\eta A/8\kappa}.
\end{align*}
It follows that for $q>2$,
\begin{align*}
\Enone\sup_{0\leq t\leq T}f(t)^{2q} & = \int_0^\infty qA^{q-1} \Pnone\bigg\{\sup_{0\leq t\leq T}f(t)^2\geq A\bigg\}dA\\
&\leq \int_0^\infty qA^{q-1}\frac{4\kappa T}{A}e^{-\eta A/8\kappa}dA\\
&=C(T,q)\kappa\int_0^\infty A^{q-2}e^{-\eta A/8\kappa}dA\\
&=C(T,q) \frac{\kappa^q}{\eta^{q-1}},
\end{align*}
which completes the proof.
\end{proof}
With Lemma~\ref{lem:limit:zeromass} in hand, we are ready to give the proof of Proposition~\ref{prop:limit:zeromass:lipschitz:1}.
\begin{proof}[Proof of Proposition~\ref{prop:limit:zeromass:lipschitz:1}] We only have to prove the result for $q>0$ sufficiently large, say $q\geq q_1$. As if it holds for $q_1$, then for every $q<q_1$, by Holder's inequality, we have
\begin{align*}
\Enone\big[m^q\sup_{0\leq t\leq T}|v(t)|^q\big]\leq\Big( \Enone\big[m^{q_1}\sup_{0\leq t\leq T}|v(t)|^{q_1}\big]\Big)^{q/q_1}\to 0,\quad\text{as}\quad m\to 0.
\end{align*}
We begin by noting that $v(t)$ from~\eqref{eqn:GLE-Markov} is written as
\begin{align*}
m\,v(t)&= me^{-\frac{\gamma}{m}t}v(0)-\int_0^te^{-\frac{\gamma}{m}(t-r)}\Phi'(x(r))dr-\sum_{k\geq 1}\sqrt{c_k}\int_0^t e^{-\frac{\gamma}{m}(t-r)}z_k(r)dr\\
&\qquad+\sqrt{2\gamma}\int_0^te^{-\frac{\gamma}{m}(t-r)}dW_0(r).
\end{align*}
Substituting $z_k(t)$ from~\eqref{eqn:limit:zeromass:lipschitz:0}, we have
\begin{align*}
m\,v(t)&= me^{-\frac{\gamma}{m}t}v(0)-\int_0^te^{-\frac{\gamma}{m}(t-r)}\Phi'(x(r))dr-\sum_{k\geq 1}\sqrt{c_k}\int_0^t e^{-\frac{\gamma}{m}(t-r)}e^{-\lambda_k r}z_k(0)dr\\
&\qquad -\sum_{k\geq 1}c_k\int_0^t e^{-\frac{\gamma}{m}(t-r)}\int_0^r e^{-\lambda_k(r-\ell)}v(\ell)d\ell\,dr\\
&\qquad-\sum_{k\geq 1}\sqrt{2c_k\lambda_k}\int_0^t e^{-\frac{\gamma}{m}(t-r)}\int_0^re^{-\lambda_k(r-\ell)}dW_k(\ell)\,dr\\
&\qquad+\sqrt{2\gamma}\int_0^te^{-\frac{\gamma}{m}(t-r)}dW_0(r).
\end{align*}
For every $q$ sufficiently large, we invoke the assumption that $\Phi'$ is Lipschitz to estimate
\begin{equation*}
\begin{aligned}
\MoveEqLeft[2]
m^{2q}\Enone\sup_{0\leq t\leq T}v(t)^{2q}\\
&\leq c(q,v) \bigg[ m^{2q}\Big[1+\Enone\sup_{0\leq t\leq T}x(t)^{2q}+\big|\sum_{k\geq 1}\sqrt{c_k}z_k\big|^{2q}\\
&\qquad +c(T)\big|\sum_{k\geq 1}c_k\big|^{2q}\int_0^T\Enone\sup_{0\leq r\leq t}v(r)^{2q}dt\\
&\qquad +\big(\sum_{k\geq 1}c_k^{(1/2-1/2q)q*}\big)^{2q/q*}\sum_{k\geq 1}\Enone\sup_{0\leq t\leq T}\Big|\sqrt{2c_k^{1/q}\lambda_k}\int_0^t e^{-\lambda_k(t-r)}dW_k(r)\Big|^{2q}\Big]\\
&\qquad +\Enone\sup_{0\leq t\leq T}\Big|\sqrt{2\gamma}\int_0^t e^{-\frac{\gamma}{m}(t-r)}dW_0(r)\Big|^{2q}\bigg],
\end{aligned}
\end{equation*}
where in the third line, we have used Holder's inequality with $\frac{1}{q^*}+\frac{1}{2q}=1$. Also, note that from the first equation of~\eqref{eqn:GLE-Markov}, it holds that
\begin{align*}
\Enone\sup_{0\leq t\leq T}x(t)^{2q}\leq c(q)\Big(x^{2q}+\int_0^T\Enone\sup_{0\leq r\leq t}v(r)^{2q}dt\Big),
\end{align*}
and that by Lemma~\ref{lem:limit:zeromass}, we have
\begin{align*}
\sum_{k\geq 1}\Enone\sup_{0\leq t\leq T}\Big|\sqrt{2c_k^{1/q}\lambda_k}\int_0^t e^{-\lambda_k(t-r)}dW_k(r)\Big|^{2q}\leq c(T,q)\sum_{k\geq 1}c_k\lambda_k,
\end{align*}
and
\begin{align*}
\Enone\sup_{0\leq t\leq T}\Big|\sqrt{2\gamma}\int_0^t e^{-\frac{\gamma}{m}(t-r)}dW_0(r)\Big|^{2q}\leq c(T,q)\gamma m^{q-1}.
\end{align*}
Furthermore, recalling $c_k$ from~\eqref{c-k}, we see that for $q>0$ sufficiently large
\begin{align*}
\sum_{k\geq 1}c_k^{(1/2-1/2q)q*}=\sum_{k\geq 1}\frac{1}{k^{(1+\alpha\beta)(1/2-1/2q)q*}}<\infty,
\end{align*}
thanks to the fact that $q*>1$ and $\alpha\beta>1$, where the latter follows from the conditions about $\alpha,\,\beta$ in Assumption~\ref{cond:wellposed}. Also, recalling $\lambda_k$ from~\eqref{c-k} and the norm $\|\cdot\|_{-s}$ from~\eqref{eqn:H_p}, it is straightforward to verify that the sums $\sum_{k\geq 1}\sqrt{c_k}z_k$, $\sum_{k\geq 1}c_k$, and $\sum_{k\geq 1}c_k\lambda_k$ are absolutely convergent. Putting everything together, we find
\begin{align*}
m^{2q}\Enone\sup_{0\leq t\leq T}v(t)^{2q}&\leq c(T,q,X(0))\Big[m^{2q}+m^{q-1}+m^{2q}\int_0^T\Enone\sup_{0\leq r\leq t}v(r)^{2q}dt\Big],
\end{align*}
where $c(T,q,X(0))>0$ is independent with $m$. Gronwall's inequality now implies
\begin{align*}
m^{2q}\Enone\sup_{0\leq t\leq T}v(t)^{2q}\leq c(T,q,X(0))(m^{2q}+m^{q-1})\to 0,\quad m\to 0.
\end{align*}
The proof is thus complete.
\end{proof}

We now turn our attention to Theorem~\ref{thm:limit:zeromass}. The proof is a slightly modification from that of Theorem 2.4 of~\cite{herzog2016small}. The key observation is that instead of controlling the exiting time of the process $x(t)$ as $m\to 0$, we are able to control $u(t)$ since $u(t)$ is independent of $m$. 
\begin{proof}[Proof of Theorem~\ref{thm:limit:zeromass}]
For $R,\,m>0$, define the following stopping times
\begin{equation} \label{eqn:stoppingtime:zeromass}
\sigma^R = \inf_{t\geq 0}\{|u(t)|\geq R\},\quad\text{ and }\quad\sigma^R_m = \inf_{t\geq 0}\{|x(t)|\geq R\},
\end{equation}
and recall 
\begin{align*}
\tau^R = \inf_{t\geq 0}\{\|U(t)\|_{\Hsdot}\geq R\},\quad\text{ and }\quad\tau^R_m = \inf_{t\geq 0}\{\|X(t)\|_{\Hs}\geq R\}.
\end{align*}
By the definitions of the norms in $\Hsdot$, cf.~\eqref{eqn:Hsdot}, we see that $\tau^R\leq \sigma^R$ a.s. From the proof of Proposition~\ref{prop:wellposed:limit:zeromass}, it is straight forward to verify that for all $T>0$,
\begin{align*}
\P{\sigma^R< T}&\leq \P{\tau^R< T} \to 0,\quad R\to \infty.
\end{align*}
For $R,\,T,\,m,\,\xi>0$, we have
\begin{align*}
\Pnone\Big\{\sup_{0\leq t\leq T}|x(t)-u(t)|>\xi\Big\}&\leq \Pnone\Big\{\sup_{0\leq t\leq T}|x(t)-u(t)|>\xi,\sigma^R\mi\sigma^R_m\geq T\Big\}\\
&\qquad+\P{\sigma^R\mi\sigma^R_m<T}.
\end{align*}
To control the first term on the above RHS, we note that for $0\leq t\leq \sigma^R\mi\sigma^R_m$, $u(t)=u^R(t)$ and $x(t)=x^R(t)$ a.s. We thus obtain the bound
\begin{align*}
\Pnone\Big\{\sup_{0\leq t\leq T}|x(t)-u(t)|>\xi,\sigma^R\mi\sigma^R_m\geq T\Big\}&\leq \Pnone\Big\{\sup_{0\leq t\leq T}|x^R(t)-u^R(t)|>\xi\Big\}\to 0,\quad m\to 0,
\end{align*} 
where the last convergence in probability follows immediately from Proposition~\ref{prop:limit:zeromass:lipschitz}. We are left to estimate $\P{\sigma^R\mi\sigma^R_m<T}$. To this end, we have that
\begin{align*}
\P{\sigma^R\mi\sigma^R_m<T}&\leq \Pnone\Big\{\sup_{0\leq t\leq T}|x^R(t)-u^R(t)|\leq \xi,\sigma^R\mi\sigma^R_m < T\Big\}\\
&\qquad+\Pnone\Big\{\sup_{0\leq t\leq T}|x^R(t)-u^R(t)|> \xi\Big\}\\
&\leq  \Pnone\Big\{\sup_{0\leq t\leq T}|x^R(t)-u^R(t)|\leq\xi,\sigma^R_m < T\leq\sigma^R\Big\}+\P{\sigma^R< T}\\
&\qquad +\Pnone\Big\{\sup_{0\leq t\leq T}|x^R(t)-u^R(t)|> \xi\Big\}.
\end{align*}
Note that for $R>1$ and $\xi\in(0,1)$, a chain of event implications is derived as follows.
\begin{align*}
\MoveEqLeft[2]\Big\{\sup_{0\leq t\leq T}|x^R(t)-u^R(t)|\leq \xi,\sigma^R_m < T\leq\sigma^R\Big\}\\
&= \Big\{\sup_{0\leq t\leq T}|x^R(t)-u(t)|\leq \xi,\sup_{0\leq t\leq T}|x^R(t)|\geq R,\sigma^R_m < T\leq\sigma^R\Big\}\\
&\subseteq \Big\{\sup_{0\leq t\leq T}|u(t)|> R-1,\sigma^R_m < T\leq\sigma^R\Big\}\\
&\subseteq \{\sigma^{R-1}< T\} ,
\end{align*}
which implies that
\begin{align*}
\Pnone\Big\{\sup_{0\leq t\leq T}|x^R(t)-u^R(t)|<\xi,\sigma^R_m < T\leq\sigma^R\Big\}\leq \P{\sigma^{R-1}< T}.
\end{align*}
Finally, putting everything together, for $R>1>\xi>0$, $T,\,m>0$, we obtain the estimate
\begin{align*}
\MoveEqLeft[2]\Pnone\Big\{\sup_{0\leq t\leq T}|x(t)-u(t)|>\xi\Big\}\\&\leq 2\Pnone\Big\{\sup_{0\leq t\leq T}|x^R(t)-u^R(t)|>\xi\Big\}+\P{\sigma^{R-1}< T}+\P{\sigma^{R} < T}\\
&\leq 2\Pnone\Big\{\sup_{0\leq t\leq T}|x^R(t)-u^R(t)|>\xi\Big\}+\P{\tau^{R-1}< T}+\P{\tau^{R} < T}.
\end{align*}
By taking $R$ sufficiently large and then shrinking $m$ further to zero, we obtain the result, thus completing the proof.
\end{proof}

\section{White-noise limit} \label{sec:limit:whitenoise}
For notation simplicity, in this section, we shall omit the subscript $\epsilon$ in 
\begin{align*}
X_\epsilon(t)=(x_\epsilon(t),v_\epsilon(t),z_{1,\epsilon}(t),\dots).
\end{align*}

With regards to the well-posedness of~\eqref{eqn:GLE-Markov:limit:whitenoise}, recalling $c_k,\,\lambda_k$ from~\eqref{c-k}, we see that the noise term is well-defined thanks to Condition (D) of Assumption~\ref{cond:wellposed}, namely
\begin{align} \label{ineq:limit:whitenoise:1}
\Enone\Big(\int_0^T\sum_{k\geq 1}\sqrt{\frac{2c_k}{\lambda_k}} dW_k(t)\Big)^2=2T\sum_{k\geq 1} \frac{c_k}{\lambda_k}=2T\sum_{k\geq 1}\frac{1}{k^{1+(\alpha-1)\beta}}<\infty.
\end{align}
The solution $(u(t),p(t))$ of~\eqref{eqn:GLE-Markov:limit:whitenoise} then is constructed using similar arguments as in the proof of Proposition~\ref{prop:wellposed:limit:zeromass} in Section~\ref{sec:limit:zeromass} via stopping times $\tau^R$, $R>0$, given by
\begin{align}
\tau^R=\inf_{t\geq 0}\{u(t)^2+p(t)^2\geq R^2\},
\end{align}
and the local solutions
\begin{equation}\label{eqn:GLE-Markov:limit:whitenoise:local}
\begin{aligned}
d u^R(t) &= p^R(t)\, d t ,\\
m\, d p^R(t)&=\Big(\!\!-\Big(\gamma+\sum_{k\geq 1}\frac{c_k}{\lambda_k}\Big) p^R(t)-\Phi'(u^R(t))\theta_R(u^R(t))\Big)\,dt\\
&\qquad\qquad\qquad\qquad\qquad\quad+\sum_{k\geq 1}\sqrt{\frac{2c_k}{\lambda_k}}dW_k(t)+\sqrt{2\gamma}\, dW_0(t),
\end{aligned}	
\end{equation}
where $\theta^R$ is the cut-off function defined in~\eqref{defn:theta-R}. Furthermore, we have the following bound: for every $T>0$ and $(u_0,p_0)\in\rbb^2$, it holds that
\begin{align}\label{ineq:whitenoise:limit:bound}
\Enone\Big[\sup_{0\leq t\leq T}u(t)^2+p(t)^2\Big]\leq C(T,u_0,p_0).
\end{align}
This estimate will be useful later in the proof of Theorem~\ref{thm:limit:whitenoise:L1converge}. The solution $X_\epsilon(t)$ is constructed using the stopping time $\tau^R_\epsilon$ given by
\begin{align}
\tau^R_\epsilon=\inf_{t\geq 0}\{\|X(t)\|_{\Hs}\geq R\},
\end{align}
and the local solutions of the cut-off system obtained from~\eqref{eqn:GLE-Markov:whitenoise:epsilon} 
\begin{equation}\label{eqn:GLE-Markov:whitenoise:epsilon:local}
\begin{aligned}
d x^R(t) &= v^R(t)\, d t ,\\
m\, d v^R(t)&=\Big(\!\!-\gamma v^R(t)-\Phi'(x^R(t))\theta^R(x^R(t))-\sum_{k\geq 1} \sqrt{\frac{c_k}{\epsilon}} z^R_{k}(t)\Big)\,dt\\
&\qquad\qquad\qquad+\sqrt{2\gamma}\, dW_0(t), \\
d z^R_{k}(t)&=\Big(\!\!-\frac{\lambda_k}{\epsilon} z^R_{k}(t)+ \sqrt{\frac{c_k}{\epsilon}}v^R(t)\Big) \, dt+\sqrt{\frac{2\lambda_k}{\epsilon}}\, dW_k(t),\qquad k\geq 1.
\end{aligned}	
\end{equation}

We now turn to the proof of Theorem~\ref{thm:limit:whitenoise:probconverge}. Similar to the proof of Theorem~\ref{thm:limit:zeromass}, it will make use of the local solutions $(u^R(t),p^R(t))$ from~\eqref{eqn:GLE-Markov:limit:whitenoise:local} and $(x^R(t),v^R(t))$ from~\eqref{eqn:GLE-Markov:whitenoise:epsilon:local}. As mentioned previously in Section~\ref{sec:limit:zeromass}, the idea essentially consists of two major steps: first, fixing $R>0$, we show that the corresponding local solution $(x^R(t),v^R(t))$ in~\eqref{eqn:GLE-Markov:whitenoise:epsilon:local} converges to $(u^R(t),p^R(t))$ in~\eqref{eqn:GLE-Markov:limit:whitenoise:local}. Then, taking $R$ sufficiently large, we obtain the convergence in probability of the original solutions by using appropriate bounds on stopping times when $(u(t),p(t))$ exits the ball of radius $R$ centered at origin. We begin by the following important result giving a uniform bound on the pair $(x(t),v(t))$.

\begin{proposition} \label{prop:limit:whitenoise:L2bound} Suppose that $\alpha,\,\beta,\, s$ satisfy Condition (D) of Assumption~\ref{cond:wellposed}. We assume further that either 

(a) $\Phi'$ is globally Lipschitz,

\noindent or

(b) $\Phi'$ is not Lipschitz, but $\Phi$ satisfies Assumptions~\ref{cond:Phi} and Assumption~ \ref{cond:Phi:whitenoise}, and $\alpha,\,\beta$ satisfy Assumption~\ref{cond:whitenoise:L1converge}.

\noindent Let $X(t)$ solve~\eqref{eqn:GLE-Markov:whitenoise:epsilon} with initial conditions $(x(0),v(0),z_{1}(0),\dots)\in\Hs$. Then, for every $T>0$, there exists a finite constant $C(T,X(0))$ such that
\begin{displaymath}
\sup_{\epsilon} \Enone\Big[\sup_{0\leq t\leq T}x(t)^2+v(t)^2\Big] \leq C(T,X(0)).
\end{displaymath}
\end{proposition}
\begin{proof} We begin by applying Duhamel's formula on $z_k(t)$ from~\eqref{eqn:GLE-Markov:whitenoise:epsilon} to see that
\begin{align}\label{eqn:limit:whitenoise:1}
z_k(t)=e^{-\frac{\lambda_k}{\epsilon}t}z_k(0)+\sqrt{\frac{c_k}{\epsilon}}\int_0^t e^{-\frac{\lambda_k}{\epsilon}(t-r)}v(r)dr+\sqrt{\frac{2\lambda_k}{\epsilon}}\int_0^t e^{-\frac{\lambda_k}{\epsilon}(t-r)}dW_k(r).
\end{align}
Substituting $z_k$ by the formula above in the second equation from~\eqref{eqn:GLE-Markov:whitenoise:epsilon} in integral form, we obtain
\begin{equation}\label{eqn:limit:whitenoise:1a}
\begin{aligned}
mv(t)&=mv(0)+\int_0^t\!\!-\gamma v(r)-\Phi'(x(r))dr+\sqrt{2\gamma}\int_0^tdW_0(r)\\
&\qquad-\sum_{k\geq 1}\sqrt{\frac{c_k}{\epsilon}} \int_0^te^{-\frac{\lambda_k}{\epsilon}r}z_k(0)dr-\sum_{k\geq 1}\frac{c_k}{\epsilon}\int_0^t\int_0^r e^{-\frac{\lambda_k}{\epsilon}(r-\ell)}v(\ell)d\ell dr\\
&\qquad-\sum_{k\geq 1}\frac{\sqrt{2c_k\lambda_k}}{\epsilon}\int_0^t\int_0^r e^{-\frac{\lambda_k}{\epsilon}(r-\ell)}dW_k(\ell)dr.
\end{aligned}
\end{equation}
It is important to note that using integration by parts, the last noise term above can be written as
\begin{multline}\label{eqn:limit:whitenoise:integrationbypart}
-\frac{\sqrt{2c_k\lambda_k}}{\epsilon}\int_0^t\int_0^r e^{-\frac{\lambda_k}{\epsilon}(r-\ell)}dW_k(\ell)dr\\= \sqrt{\frac{2c_k}{\lambda_k}}\int_0^te^{-\frac{\lambda_k}{\epsilon}(t-r)}dW_k(r)-\sqrt{\frac{2c_k}{\lambda_k}}\int_0^tdW_k(r).
\end{multline}

Suppose that Condition (\emph{a}) holds, i.e., $\Phi'$ is Lipschitz. In view of~\eqref{eqn:limit:whitenoise:1a} and~\eqref{eqn:limit:whitenoise:integrationbypart}, we have the following estimate for every $q>1$ and $0\leq t\leq T$,
\begin{equation*}\label{ineq:limit:whitenoise:2}
\begin{aligned}
v(t)^{2q}
&\leq c(q)\bigg[ |v(0)|^{2q}+\int_0^T\!\!\sup_{0\leq r\leq t}v(r)^{2q}+x(r)^{2q}dr+ \Big(\sum_{k\geq 1}\frac{\sqrt{\epsilon c_k}}{\lambda_k}(1-e^{-\frac{\lambda_k}{\epsilon}t})|z_k(0)|\Big)^{2q}\\
&\qquad+ c(T)\Big(\sum_{k\geq 1}\frac{c_k}{\lambda_k}\Big)^{2q}\int_0^T \sup_{0\leq \ell\leq r}v(\ell)^{2q}dr\\
&\qquad+\sup_{0\leq t\leq T}\Big(\sqrt{2\gamma}\int_0^tdW_0(r)-\sum_{k\geq 1}\sqrt{\frac{2c_k}{\lambda_k}}\int_0^tdW_k(r)\Big)^{2q}\\
&\qquad+\Big(\sum_{k\geq 1}k^{-sq*}\Big)^{2q/q*}\sum_{k\geq 1}\sup_{0\leq t\leq T}\Big|\sqrt{\frac{2c_k k^{2s}}{\lambda_k}}\int_0^te^{-\frac{\lambda_k}{\epsilon}(t-r)}dW_k(r)\Big|^{2q}\bigg],
\end{aligned}
\end{equation*}
where in the last line, we have used Holder's inequality with $\frac{1}{2q}+\frac{1}{q*}=1$. Note that for every $x\geq 0$, we have $1-e^{-x}\leq \sqrt{x}$. Using this inequality, we estimate 
\begin{equation}\label{ineq:limit:whitenoise:2a}
\begin{aligned}
\sum_{k\geq 1}\frac{\sqrt{\epsilon c_k}}{\lambda_k}(1-e^{-\frac{\lambda_k}{\epsilon}t})|z_k(0)|&\leq \sum_{k\geq 1}\sqrt{\frac{ c_k t}{\lambda_k}}|z_k(0)|\\
&\leq\Big(T\sum_{k\geq 1}\frac{c_k k^{2s}}{\lambda_k}\sum_{k\geq 1}k^{-2s}z_k(0)^2\Big)^{1/2}.
\end{aligned}
\end{equation}
Recalling $c_k,\,\lambda_k$ from~\eqref{c-k} and the norm $\|\cdot\|_{\Hs}$ from~\eqref{eqn:H_p}, thanks to Condition (D) of Assumption~\ref{cond:wellposed}, we see that the above RHS is finite and so is the sum $\sum_{k\geq 1}c_k/\lambda_k$. In addition, using Burkholder-Davis-Gundy's inequality, we have
\begin{equation} \label{ineq:limit:whitenoise:2b}
\begin{aligned}
\MoveEqLeft[5]\Enone\sup_{0\leq t\leq T}\Big(\sqrt{2\gamma}\int_0^tdW_0(r)-\sum_{k\geq 1}\sqrt{\frac{2c_k}{\lambda_k}}\int_0^tdW_k(r)\Big)^{2q}\\
&\leq c(q)\Enone\Big(2\gamma\int_0^Tdr+\sum_{k\geq 1}\frac{2c_k}{\lambda_k}\int_0^Tdr\Big)^{q}=c(T,q)<\infty.
\end{aligned}
\end{equation}
Finally, we invoke Lemma~\ref{lem:limit:zeromass} again to find
\begin{equation}\label{ineq:limit:whitenoise:4}
\begin{aligned}
\MoveEqLeft[3]
\Enone\sum_{k\geq 1}\sup_{0\leq t\leq T}\Big|\sqrt{\frac{2c_k k^{2s}}{\lambda_k}}\int_0^te^{-\frac{\lambda_k}{\epsilon}(t-r)}dW_k(r)\Big|^{2q}\\
&\leq c(T,q)\sum_{l\geq 1}\frac{\epsilon^{q-1}c_k^q k^{2sq}}{\lambda_k^{2q-1}}\\
&=c(T,q)\epsilon^{q-1}\sum_{k\geq 1}\frac{1}{k^{q+(q\alpha-2q+1)\beta-2sq}}.
\end{aligned}
\end{equation}
Note that for $\alpha>1$, $s>1/2$ and $\beta>0$ satisfying Condition (D) of Assumption~\ref{cond:wellposed}, there exist constants $q>1$ and $0<q*<2$ such that 
\begin{align*}
q+(q\alpha-2q+1)\beta-2sq>1,\quad sq*>1,\quad\text{and}\quad\frac{1}{2q}+\frac{1}{q*}=1.
\end{align*}
Consequently, the sums $\sum_{k\geq 1}k^{-[q+(q\alpha-2q+1)\beta-2sq]}$ and $\sum_{k\geq 1}k^{-sq*}$ are both finite. Combining everything together, we infer
\begin{align*}
\Enone\big(\sup_{0\leq t\leq T}x(t)^{2q}+v(t)^{2q}\big)&\leq c(T,q,X(0))\Big[1+\int_0^T\Enone\big(\sup_{0\leq r\leq t}x(r)^{2q}+v(r)^{2q}\big) dt\Big].
\end{align*}
 Choosing such $q$, we finally obtain the following estimate using Gronwall's inequality
\begin{align}\label{ineq:limit:whitenoise:5}
\Enone\big(\sup_{0\leq t\leq T}x(t)^{2q}+v(t)^{2q}\big)&\leq c(T,q,X(0)),
\end{align}
which proves the result for Condition (\emph{a}) since $q>1$. 

Now suppose that Condition (\emph{b}) holds. To simplify notation, we set
\begin{align*}
g_k(t):=\int_0^t e^{-\frac{\lambda_k}{\epsilon}(t-r)}v(r)dr,\quad\text{and}\quad w_k(t):= \sqrt{2}\int_0^te^{-\frac{\lambda_k}{\epsilon}(t-r)}dW_k(r).
\end{align*}
Following ~\eqref{eqn:limit:whitenoise:1a} and~\eqref{eqn:limit:whitenoise:integrationbypart}, the equation on $v(t)$ is written as
\begin{align*}
d\Big(mv(t)-\sum_{k\geq 1}\sqrt{\frac{c_k}{\lambda_k}}w_k(t)\Big)& = \Big(-\gamma v(t)-\Phi'(x(t))-\sum_{k\geq 1}\sqrt{\frac{c_k}{\epsilon}}e^{-\frac{\lambda_k}{\epsilon}t}z_k(0)-\sum_{k\geq 1}\frac{c_k}{\epsilon}g_k(t)\Big)dt\\
&\qquad+\sqrt{2\gamma}dW_0(t)-\sum_{k\geq 1}\sqrt{ \frac{2c_k}{\lambda_k}}dW_k(t).
\end{align*}
We apply Ito's formula to $\big(v(t)-\frac{1}{m}\sum_{k\geq 1}\sqrt{\frac{c_k}{\lambda_k}}w_k(t)\big)^2/2+\Phi(x(t))/m$ to see that
\begin{align*}
\MoveEqLeft[1]d\Big[\Big(v(t)-\frac{1}{m}\sum_{k\geq 1}\sqrt{\frac{c_k}{\lambda_k}}w_k(t)\Big)^2/2+\Phi(x(t))/m\Big]\\
&=\Big(v(t)-\frac{1}{m}\sum_{k\geq 1}\sqrt{\frac{c_k}{\lambda_k}}w_k(t)\Big)\Big(-\frac{\gamma}{m} v(t)-\sum_{k\geq 1}\sqrt{\frac{c_k}{m^2\epsilon}}e^{-\frac{\lambda_k}{\epsilon}t}z_k(0)-\sum_{k\geq 1}\frac{c_k}{m\epsilon}g_k(t)\Big)dt\\
&\qquad+\Big(v(t)-\frac{1}{m}\sum_{k\geq 1}\sqrt{\frac{c_k}{\lambda_k}}w_k(t)\Big)\Big(\frac{\sqrt{2\gamma}}{m}dW_0(t)-\sum_{k\geq 1}\sqrt{ \frac{2c_k}{m^2\lambda_k}}dW_k(t)\Big)\\
&\qquad +\Big(\frac{\Phi'(x(t))}{m}\sum_{k\geq 1}\sqrt{\frac{c_k}{\lambda_k}}w_k(t)+\frac{\gamma}{m^2}+\sum_{k\geq 1}\frac{c_k}{\lambda_k}\Big) dt.
\end{align*}
We proceed to estimate the above RHS. Firstly, we invoke estimate~\eqref{ineq:limit:whitenoise:2a} to find
\begin{align*}
\MoveEqLeft[4]\int_0^t \Big(v(r)-\frac{1}{m}\sum_{k\geq 1}\sqrt{\frac{c_k}{\lambda_k}}w_k(r)\Big)\Big(-\sum_{k\geq 1}\sqrt{\frac{c_k}{m^2\epsilon}}e^{-\frac{\lambda_k}{\epsilon}r}z_k\Big) dr \\
&\leq \sup_{0\leq r\leq t} \Big|v(r)-\frac{1}{m}\sum_{k\geq 1}\sqrt{\frac{c_k}{\lambda_k}}w_k(r)\Big|\,\sum_{k\geq 1}\frac{\sqrt{\epsilon c_k}}{\lambda_k}(1-e^{-\frac{\lambda_k}{\epsilon}t})|z_k|\\
&\leq \sup_{0\leq r\leq t} \Big|v(r)-\frac{1}{m}\sum_{k\geq 1}\sqrt{\frac{c_k}{\lambda_k}}w_k(r)\Big|\,\Big(\sum_{k\geq 1}\frac{c_k k^{2s}}{\lambda_k}\sum_{k\geq 1}k^{-2s}z_k^2\Big)^{1/2}\\
&\leq \frac{1}{2}\sup_{0\leq r\leq t} \Big|v(r)-\frac{1}{m}\sum_{k\geq 1}\sqrt{\frac{c_k}{\lambda_k}}w_k(r)\Big|^2+2\sum_{k\geq 1}\frac{c_k k^{2s}}{\lambda_k}\sum_{k\geq 1}k^{-2s}z_k^2.
\end{align*}
Similarly, we have
\begin{align*}
g_k(r)=\int_0^r e^{-\frac{\lambda_k}{\epsilon}(r-\ell)}v(\ell)d\ell\leq\sup_{0\leq \ell\leq r}|v(\ell)|\frac{\epsilon}{\lambda_k},
\end{align*}
which implies that
\begin{align*}
\MoveEqLeft[4]\int_0^t\Big(v(r)-\frac{1}{m}\sum_{k\geq 1}\sqrt{\frac{c_k}{\lambda_k}}w_k(r)\Big)\Big(-\sum_{k\geq 1}\frac{c_k}{m\epsilon}g_k(r)\Big)dr\\
&\leq c\Big( \sum_{k\geq 1}\frac{c_k}{\lambda_k}\Big)\int_0^t\sup_{0\leq \ell\leq r }v(\ell)^2+\sup_{0\leq \ell\leq r }\Big(\sum_{k\geq 1}\sqrt{\frac{c_k}{\lambda_k}}w_k(\ell)\Big)^2dr.
\end{align*}
With regard to the martingale term, we invoke Burkholder-Davis-Gundy's inequality to estimate
\begin{align*}
\MoveEqLeft[4]\Enone\sup_{0\leq r\leq t }\Big|\int_0^r\Big(v(\ell)-\frac{1}{m}\sum_{k\geq 1}\sqrt{\frac{c_k}{\lambda_k}}w_k(\ell)\Big)\Big(\frac{\sqrt{2\gamma}}{m}dW_0(\ell)-\sum_{k\geq 1}\sqrt{ \frac{2c_k}{m^2\lambda_k}}dW_k(\ell)\Big)\Big|\\
&\leq c\Big[\Big(\frac{2\gamma}{m^2}+\sum_{k\geq 1}\frac{2c_k}{m^2\lambda_k}\Big)\int_0^t\Enone\Big(v(r)-\frac{1}{m}\sum_{k\geq 1}\sqrt{\frac{c_k}{\lambda_k}}w_k(r)\Big)^2dr+1\Big].
\end{align*}
Lastly, we employ Assumption~\ref{cond:Phi:whitenoise} to infer
\begin{align*}
\int_0^t \Phi'(x(r))\sum_{k\geq 1}\sqrt{\frac{c_k}{\lambda_k}}w_k(r)dr&\leq c\int_0^t\Phi(x(r))+\Big(\sum_{k\geq 1}\sqrt{\frac{c_k}{\lambda_k}}w_k(r)\Big)^n+1\, dr.
\end{align*}
Putting everything together, we arive at the following inequality
\begin{align*}
\Enone\sup_{0\leq t\leq T } v(t)^2+\Phi(x(t))& \leq c(T)\Big[1+\int_0^T\Enone\sup_{0\leq r\leq t } v(r)^2+\Phi(x(r))\,dt\\
&\quad+\Enone\sup_{0\leq t\leq T}\Big(\sum_{k\geq 1}\sqrt{\frac{c_k}{\lambda_k}}w_k(r)\Big)^2  +\Enone\sup_{0\leq t\leq T}\Big(\sum_{k\geq 1}\sqrt{\frac{c_k}{\lambda_k}}w_k(r)\Big)^n\Big].
\end{align*}
The result now follows immediately from Gronwall's inequality if we can show that the last two terms on the above RHS is finite and independent of $\epsilon$. To this end, we claim that for every $T>0$ and $q>2$, there exists a finite constant $C(T,q)>0$ such that
\begin{equation}\label{ineq:limit:whitenoise:6}
\Enone\sup_{0\leq t\leq T}\Big(\sum_{k\geq 1}\sqrt{\frac{c_k}{\lambda_k}}w_k(r)\Big)^{2q} \leq c(T,q).
\end{equation}
Recalling $w_k(t):= \sqrt{2}\int_0^te^{-\frac{\lambda_k}{\epsilon}(t-r)}dW_k(r)$, similar to \eqref{ineq:limit:whitenoise:4}, we employ Holder's inequality and Lemma~\ref{lem:limit:zeromass} to see that
\begin{align*}
\MoveEqLeft[4]\Enone\sup_{0\leq t\leq T}\Big(\sum_{k\geq 1}\sqrt{\frac{c_k}{\lambda_k}}w_k(r)\Big)^{2q}\\
&\leq\Big(\sum_{k\geq 1}k^{-q_1q*}\Big)^{2q/q*}\sum_{k\geq 1}\Enone\bigg[\sup_{0\leq t\leq T}\Big|\sqrt{\frac{c_k k^{2q_1}}{\lambda_k}}w_k(t)\Big|^{2q}\bigg]\\
&=c(T,q)\Big(\sum_{k\geq 1}k^{-q_1q*}\Big)^{2q/q*}\epsilon^{q-1}\sum_{k\geq 1}\frac{1}{k^{q+(q\alpha-2q+1)\beta-2q_1q}}.
\end{align*}
where $\frac{1}{2q}+\frac{1}{q*}=1$ and $q_1>0$ is a constant satisfying
\begin{align*}
q_1q*>1\quad\text{and}\quad q+(q\alpha-2q+1)\beta-2q_1q>1.
\end{align*}
Solving the above inequalities for $q_1$, we find
\begin{align*}
\frac{1+(\alpha-2)\beta}{2}+\frac{\beta}{2q}-\frac{1}{2q}>q_1>1-\frac{1}{2q},
\end{align*}
which is always possible thanks to the second part of Condition (b), namely, $(\alpha-2)\beta>1$. The proof is thus complete.
\end{proof}

\begin{remark} \label{ref:whitenoise} (a) The trick of using integration by part in~\eqref{eqn:limit:whitenoise:integrationbypart} was previously employed in ~\cite{hottovy2015smoluchowski,ottobre2011asymptotic}.

(b) The condition $\alpha>1$ of the diffusive regime was employed throughout the proof of Proposition~\ref{prop:limit:whitenoise:L2bound}, e.g. estimates~\eqref{ineq:limit:whitenoise:2a},~\eqref{ineq:limit:whitenoise:2b} and ~\eqref{ineq:limit:whitenoise:4}.
\end{remark}

%
%

\begin{proposition}\label{prop:limit:whitenoise:lipschitz} Under the same Hypothesis of Theorem~\ref{thm:limit:whitenoise:probconverge}, assume further that $\Phi'(x)$ is globally Lipschitz. Let $X_\epsilon(t)=(x_\epsilon(t),v_\epsilon(t),z_{1,\epsilon}(t),\dots)$ be the solution of~\eqref{eqn:GLE-Markov:whitenoise:epsilon} with initial conditions 
\begin{align*}(x_\epsilon(0),v_\epsilon(0),z_{1,\epsilon}(0),z_{2,\epsilon}(0),\dots)=(x,v,z_{1},z_2,\dots)\in\Hs,
\end{align*}
 and $(u(t),p(t))$ be the solution of~\eqref{eqn:GLE-Markov:limit:whitenoise} with initial conditions $(u(0),p(0))=(x,v)$. Then, for every $T>0$,
\begin{equation*}
\Enone\Big[\sup_{0\leq t\leq T}|x(t)-u(t)|^2+\sup_{0\leq t\leq T}|v(t)-p(t)|^2\Big]\to 0, \quad \epsilon\to 0.
\end{equation*}
\end{proposition}
The proof of Proposition~\ref{prop:limit:whitenoise:lipschitz} is based on that of Theorem 2.6 in~\cite{ottobre2011asymptotic}, adapted to our infinite-dimensional setting. 
\begin{proof} Setting $\xbar(t) := x(t)-u(t),\,\vbar(t):=v(t)-p(t)$, we see that from\eqref{eqn:GLE-Markov:whitenoise:epsilon}, \eqref{eqn:GLE-Markov:limit:whitenoise}, $(\xbar(t),\vbar(t))$ satisfies the following system
\begin{align*}
 \xbar(t) &=\int_0^t \!\!\vbar(r)\, d r, \\
m \vbar(t)&=\int_0^t\!\!\Big(\!\!-\gamma\vbar(r)+\Big(\sum_{k\geq 1}\frac{c_k}{\lambda_k}\Big) p(r)-\big[\Phi'(x(r))-\Phi'(u(r))\big]-\sum_{k\geq 1}\sqrt{\frac{c_k}{\epsilon}}z_k(r)\Big)\,dr\\
&\qquad+\int_0^t\sum_{k\geq 1}\sqrt{\frac{2c_k}{\lambda_k}}dW_k(r).
\end{align*}
with the initial conditions $(\xbar(0),\vbar(0))=(0,0)$. Regrading $z_k(t)$ terms, we integrate with respect to time the third equation in~\eqref{eqn:GLE-Markov:whitenoise:epsilon} to find that
\begin{align*}
\frac{\sqrt{\epsilon c_k}}{\lambda_k} (z_k(t)-z_k(0)) -\frac{c_k}{\lambda_k}\int_0^tv(r)dr-\sqrt{\frac{ c_k}{\lambda_k}}\int_0^tdW_k(r)=-\sqrt{\frac{c_k}{\epsilon}}\int_0^tz_k(r).
\end{align*}
With these observations, the system of integral equations on $(\xbar(t),\vbar(t))$ becomes
\begin{equation} \label{eqn:limit:whitenoise:lipschitz:1}
\begin{aligned}
 \xbar(t) &=\int_0^t\!\! \vbar(r)\, d r, \\
m \vbar(t)&=\int_0^t\!\!\Big(\!\!-\Big(\gamma+\sum_{k\geq 1}\frac{c_k}{\lambda_k}\Big) \vbar(r)-\big[\Phi'(x(r))-\Phi'(u(r))\big]\Big)\,dr\\
&\qquad\qquad\qquad+\sqrt{\epsilon}\sum_{k\geq 1}\frac{\sqrt{c_k}}{\lambda_k}(z_k(t)-z_k(0)).
\end{aligned}
\end{equation}
In the above system, we have implicitly re-arranged infinitely many terms, resulting in the cancellation of noise terms. Recalling $c_k,\,\lambda_k$ from~\eqref{c-k} and the norm $\|\cdot\|_{\Hs}$ from~\eqref{eqn:H_p}, this re-arrangement is possible following from~\eqref{ineq:limit:whitenoise:1} and the estimate
\begin{align*}
\sum_{k\geq 1}\Big|\frac{\sqrt{\epsilon c_k}}{\lambda_k}(z_k(t)-z_k(0))\Big|&\leq \sum_{k\geq 1}\frac{c_k}{\lambda_k}\int_0^t |v(r)|dr+\sum_{k\geq 1}+\frac{1}{\sqrt{\epsilon}}\int_0^t\sum_{k\geq 1}\sqrt{c_k}|z_k(r)|dr\\
&\qquad+\sum_{k\geq 1}\sqrt{\frac{ c_k}{\lambda_k}}\big|\int_0^tdW_k(r)\big|\\
&<\infty, \text{ a.s.}
\end{align*}
thanks to condition (D) of Assumption~\ref{cond:wellposed}. We invoke the assumption that $\Phi'$ is globally Lipschitz and Gronwall's inequality to deduce from~\eqref{eqn:limit:whitenoise:lipschitz:1}
\begin{align*}
\Enone\sup_{0\leq t\leq T}|\xbar(t)|+|\vbar(t)| \leq C(T)\sqrt{\epsilon}\,\Enone\sup_{0\leq t\leq T}\Big|\sum_{k\geq 1}\frac{\sqrt{c_k}}{\lambda_k}(z_k(t)-z_k(0))\Big|.
\end{align*}
The result now follows immediately from Proposition~\ref{prop:limit:whitenoise:lipschitz:1} below.
\end{proof}

\begin{proposition} \label{prop:limit:whitenoise:lipschitz:1} Under the same Hypothesis of Proposition~\ref{prop:limit:whitenoise:lipschitz}, suppose that  $X(t)=(x(t),v(t),z_{1}(t),\dots)$ solves~\eqref{eqn:GLE-Markov:whitenoise:epsilon} with initial conditions $(x(0),v(0),z_{1}(0),\dots)\in \Hs$.
Then,
\begin{align*}
\sqrt{\epsilon}\,\Enone\sup_{0\leq t\leq T}\Big|\sum_{k\geq 1}\frac{\sqrt{c_k}}{\lambda_k}(z_{k}(t)-z_k(0))\Big|\to 0,\quad\epsilon\to 0.
\end{align*} 
\end{proposition}
\begin{proof} From~\eqref{eqn:limit:whitenoise:1}, we see that
\begin{equation}\label{ineq:limit:whitenoise:3}
\begin{aligned}
\sum_{k\geq 1}\frac{\sqrt{\epsilon c_k}}{\lambda_k}|z_k(t)-z_k(0)|
&\leq \sum_{k\geq 1}\frac{\sqrt{\epsilon c_k}}{\lambda_k}\big(e^{-\frac{\lambda_k}{\epsilon}t}-1\big)|z_k(0)|+\sum_{k\geq 1}\frac{c_k}{\lambda_k}\int_0^t e^{-\frac{\lambda_k}{\epsilon}(t-r)}|v(r)|dr\\
&\qquad+\sum_{k\geq 1}\sqrt{\frac{2c_k}{\lambda_k}}\Big|\int_0^t e^{-\frac{\lambda_k}{\epsilon}(t-r)}dW_k(r)\Big|.
\end{aligned}
\end{equation}
We aim to show that each series on the above RHS converges to zero in expectation as $\epsilon\downarrow 0$. We note that the convergence to zero of the last series follows immediately from~\eqref{ineq:limit:whitenoise:4}. For the other two terms, we shall make use of the following inequality: for $q>0$, there exists $c(q)>0$ such that for every $x\geq 0$, it holds that
\begin{equation}\label{ineq:basic:exponential}
1-e^{-x}\leq c(q)x^q.
\end{equation} 
For a positive $q_1<\frac{1}{2}$ (to be chosen later), we estimate the first sum on the RHS of~\eqref{ineq:limit:whitenoise:3} as follows. 
\begin{align*}
\sum_{k\geq 1}\sup_{0\leq t\leq T}\frac{\sqrt{\epsilon c_k}}{\lambda_k}\big(e^{-\frac{\lambda_k}{\epsilon}t}-1\big)z_k(0)&\leq c(T,q_1)\epsilon^{1/2-q_1}\frac{c_k^{1/2}}{\lambda_k^{1-q_1}}z_k(0)\\
&\leq   c(T,q_1) \epsilon^{1/2-q_1}\Big(\sum_{k\geq 1}\frac{c_k k^{2s}}{\lambda_k^{2-2q_1}}\Big)^{1/2}\Big(\sum_{k\geq 1}k^{-2s}z_k(0)^2\Big)^{1/2},
\end{align*}
where we have used~\eqref{ineq:basic:exponential} on the first line and Holder's inequality on the second line, respectively. Recalling \eqref{c-k}, we have
\begin{align*}
\sum_{k \geq 1}\frac{c_k k^{2s}}{\lambda_k^{2-2q_1}} = \sum_{k\geq 1}\frac{1}{k^{1+(\alpha+2q_1-2)\beta-2s}}.
\end{align*}
In view of Condition (D) of Assumption~\ref{cond:wellposed}, there always exists a constant $q_1\in (0,1/2)$ such that $(2q_1-1)\beta+(\alpha-1)\beta-2s>0$, which implies that the above RHS is finite. Similarly, we have
\begin{align*}
\sum_{k\geq 1}\sup_{0\leq t\leq T}\frac{c_k}{\lambda_k}\int_0^t e^{-\frac{\lambda_k}{\epsilon}(t-r)}|v(r)|dr&\leq\sum_{k\geq 1} \sup_{0\leq t\leq T} \frac{\epsilon c_k}{\lambda_k^2}\big(1-e^{-\frac{\lambda_k}{\epsilon}t}\big)\sup_{0\leq t\leq T}|v(t)|\\
&\leq c(T,q_2)\epsilon^{1-q_2}\sum_{k\geq 1}\frac{c_k}{\lambda_k^{2-q_2}}\sup_{0\leq t\leq T}|v(t)|\\
&= c(T,q_2)\epsilon^{1-q_2}\sum_{k\geq 1}\frac{1}{k^{1+(\alpha-2+q_2)\beta}}\sup_{0\leq t\leq T}|v(t)|.
\end{align*}
We invoke Condition (D) from Assumption~\ref{cond:wellposed} again to see that there exists a positive $q_2\in(0,1)$ such that $\alpha-2+q_2>0$. Choosing such $q_2$ implies that the series on the above RHS is convergent. We thus obtain the estimate
\begin{align*}
\Enone\sum_{k\geq 1}\sup_{0\leq t\leq T}\frac{c_k}{\lambda_k}\int_0^t e^{-\frac{\lambda_k}{\epsilon}(t-r)}|v(r)|dr&\leq c(T,q_2)\epsilon^{1-q_2}\Enone\sup_{0\leq t\leq T}|v(t)|\leq c(T,q_2)\epsilon^{1-q_2},
\end{align*}
where the last implication follows from Proposition~\ref{prop:limit:whitenoise:L2bound}. Putting everything together, we obtain the result.
\end{proof}

Since we will make use of exiting times, with a slightly abuse of notation, it is convenient to recall from~\eqref{eqn:stoppingtime:zeromass} for $R>0$
\begin{equation*} \label{defn:stoppingtime:whitenoise}
\sigma^R = \inf_{t\geq 0}\{|u(t)|\geq R \},\quad\text{ and }\quad \sigma^R_\epsilon = \inf_{t\geq 0}\{|x(t)|\geq R .\}
\end{equation*}
With Proposition~\ref{prop:limit:whitenoise:lipschitz} in hand, we give the proof of Theorem~\ref{thm:limit:whitenoise:probconverge}.
\begin{proof}[Proof of Theorem~\ref{thm:limit:whitenoise:probconverge}]
The arguments are almost the same as those in the proof of Theorem~\ref{thm:limit:zeromass} and hence omitted. The only difference here is the appearance of the term $|v(t)-p(t)|$. Nevertheless, we note that for $0\leq t\leq \sigma^R\mi\sigma^R_\epsilon$,
\begin{align*}
 (u(t),p(t))=(u^R(t),p^R(t))\quad \text{ and }\quad (x(t),v(t))=(x^R(t),v^R(t)),
\end{align*}
and thus the proof of Theorem~\ref{thm:limit:zeromass} is applicable.
\end{proof}
We finally turn our attention to Theorem~\ref{thm:limit:whitenoise:L1converge}. The proof is relatively short and will make use of Condition (\emph{b}) in Proposition~\ref{prop:limit:whitenoise:L2bound}.
\begin{proof}[Proof of Theorem~\ref{thm:limit:whitenoise:L1converge}] 
For given $R>0$, let $\sigma^R,\,\sigma^R_\epsilon$ be defined as in~\eqref{eqn:stoppingtime:zeromass}. As mentioned above, for $0\leq t\leq \sigma^R\mi\sigma^R_\epsilon$,
\begin{align*}
 (u(t),p(t))=(u^R(t),p^R(t))\quad \text{ and }\quad (x(t),v(t))=(x^R(t),v^R(t)).
\end{align*} 
We then have a chain of implications
\begin{equation*}
\begin{aligned}
\MoveEqLeft[3]
\Enone\Big[\sup_{0\leq t\leq T}|x(t)-u(t)|^q+|v(t)-p(t)|^q \Big] \\
&=
\Enone\Big[\Big(\sup_{0\leq t\leq T}|x(t)-u(t)|^q+|v(t)-p(t)|^q \Big)1_{\{\sigma^R\mi\sigma^R_\epsilon<T\}}\Big]\\
&\qquad\qquad\qquad
+\Enone\Big[\Big(\sup_{0\leq t\leq T}|x(t)-u(t)|^q+|v(t)-p(t)|^q \Big)1_{\{\sigma^R\mi\sigma^R_\epsilon>T\}}\Big]\\
&\leq \Enone\Big[\Big(\sup_{0\leq t\leq T}|x(t)-u(t)|^q+|v(t)-p(t)|^q \Big)1_{\{\sigma^R\mi\sigma^R_\epsilon<T\}}\Big]\\
&\qquad\qquad\qquad
+\Enone\Big[\sup_{0\leq t\leq T}|x^R(t)-u^R(t)|^q+|v^R(t)-p^R(t)|^q\Big].
\end{aligned}
\end{equation*}
On one hand, in view of Proposition~\ref{prop:limit:whitenoise:lipschitz}, since $1\leq q<2$, it holds that
\begin{align*}
\Enone\Big[\sup_{0\leq t\leq T}|x^R(t)-u^R(t)|+|v^R(t)-p^R(t)|\Big]\to 0,\quad\epsilon\to 0.
\end{align*}
On the other hand, we invoke Holder's inequality with $\frac{q}{2}+\frac{1}{q*}=1$ to estimate
\begin{multline*}
\Enone\Big[\Big(\sup_{0\leq t\leq T}|x(t)-u(t)|^q+|v(t)-p(t)|^q \Big)1_{\{\sigma^R\mi\sigma^R_\epsilon<T\}}\Big]\\
\leq c\Big(\Enone\Big[\sup_{0\leq t\leq T}x(t)^2+v(t)^2\Big]+\Enone\Big[\sup_{0\leq t\leq T}u(t)^2+p(t)^2\Big] \Big)^{q/2}\Big(\P{\sigma^R\mi\sigma^R_\epsilon<T}\Big)^{1/q*}.
\end{multline*}
Notice that by Markov's inequality, we have
\begin{align*}
\P{\sigma^R\mi\sigma^R_\epsilon<T}&\leq \P{\sigma^R<T}+\P{\sigma^R_\epsilon<T}\\
&\leq\Pnone\Big\{\sup_{0\leq t\leq T}|u(t)|\geq R\Big\}+\Pnone\Big\{\sup_{0\leq t\leq T}|x(t)|\geq R\Big\}\\
&\leq \frac{\Enone\big[\sup_{0\leq t\leq T}|u(t)|^2\big]+\Enone\big[\sup_{0\leq t\leq T}x(t)^2\big]}{R^2}.
\end{align*}
The result now follows immediately from~\eqref{ineq:whitenoise:limit:bound} and Proposition~\ref{prop:limit:whitenoise:L2bound} by first taking $R$ sufficiently large and then shrinking $\epsilon$ further to zero. The proof is thus complete.
\end{proof}

\section{Discussion} We have established rigorous results on the asymptotical analysis of an infinite-dimensional GLE when the memory kernel $K(t)$ has a power-law decay, i.e. $K(t)\sim t^{-\alpha}$ as $t\to\infty$. With regards to the small-mass limit, we are able to obtain the convergence in probability of the GLE for every exponent constant $\alpha>0$. However, in the white-noise limit, a similar convergence was established only when $\alpha>1$. The method that we employed was not able to extend the result when $\alpha\in(0,1]$, which is interestingly also the barrier for the unique ergodicity of~\eqref{eqn:GLE-Markov} \cite{glatt2018generalized}. As mentioned earlier in Remark~\ref{ref:whitenoise}, our technique in the white-nosie limit requires that the memory be integrable ($\alpha>1$) for the analysis of the solutions as well as the asymptotical behaviors. It therefore remains an open question whether the solution's energy is still bounded uniformly and there exists a limiting system. 

Finally, another question for future works is whether one can take both limits in sequence, which means that the small-mass variable $m$ is written as an order of $\epsilon$, the white-noise variable. It is not clear that the theorems presented in this work combined together are able to produce an explicit answer. We note that a similar study in finite-dimensional setting was carried out in~\cite{lim2017homogenization}. Yet, we have not been able to see if the same method can be applied to our infinite system. We believe handling this case will require a more substantial work.

\section*{Acknowledgement} The author would like to thank Nathan Glatt-Holtz, David Herzog and Scott McKinley for many fruitful discussions. The author is also grateful for support through grant NSF DMS-1644290.

\bibliographystyle{plain}
\bibliography{nonlinear-gle}

\end{document}